\documentclass[11pt]{article}
\usepackage{amsmath,amsfonts,amsthm,amssymb,fullpage}
\usepackage[dvips]{hyperref}
\makeatletter

\@addtoreset{equation}{section}
\makeatother

\newtheorem{theo}{Theorem}[section]

\newtheorem{lemm}[theo]{Lemma}
\newtheorem{prop}[theo]{Proposition}
\newtheorem{coro}[theo]{Corollary}

\newcommand{\bs}{\boldsymbol}{}
\def\pa{0}
\def\pb{1}
\def\pc{2}
\def\pd{3}
\def\pe{4}
\def\pf{5}
\def\pg{6}
\def\ph{7}
\allowdisplaybreaks[4]
\title{Duality transformation formulas for multiple elliptic hypergeometric series of type $BC$}
\author{Yasushi Komori\footnote{
		Department of Mathematics, Rikkyo University, Nishi-Ikebukuro, Toshima-ku, Tokyo 171-8501, Japan. 
		\newline \hspace{14pt} Email:\,komori@rikkyo.ac.jp
	} ,
	Yasuho Masuda\footnote{Department of Mathematics, Kobe University, Rokko, Kobe 657-8501, Japan. 
		\newline \hspace{14pt} Email:\,myasuho@math.kobe-u.ac.jp} ,
	Masatoshi Noumi\footnote{Department of Mathematics, Kobe University, Rokko, Kobe 657-8501, Japan.
		\newline \hspace{14pt} Email:\,noumi@math.kobe-u.ac.jp} }
\date{\ }

\begin{document}
\maketitle
\begin{abstract}
New duality transformation formulas are proposed for multiple elliptic hypergeometric series
of type $BC$ and of type $C$.
Various transformation and  summation formulas are derived as special cases to recover some previously known results.
\end{abstract}
\begin{quote}
\end{quote}
\section{Introduction}
Elliptic hypergeometric series were first introduced by Frenkel\,--\,Turaev \cite{FT1}
in the context of elliptic solutions of the Yang\,--\,Baxter equations.
They found elliptic extensions of the Jackson summation formula, called the Frenkel\,--\,Turaev summation formula,
and the Bailey transformation formula for terminating very well-poised
basic hypergeometric series.
As generalizations of these results,
summation and transformation formulas 
for multiple elliptic hypergeometric series have been 
studied by several authors. 
Warnaar \cite{Wa1} extended
the Frenkel\,--\,Turaev summation formula to a multiple summation formula of type $C_n$.
Also, Rosengren \cite{Ro2} investigated
summation and transformation formulas
associated with classical root systems.

A generalization of the elliptic Bailey transformation to multiple elliptic hypergeometric series,
called the \emph{duality transformation formula} of type $A_{n}$,
was obtained independently by Kajihara\,--\,Noumi \cite{KaN1} and by Rosengren \cite{Ro3}.
A characteristic feature of this formula is that 
the summations of the two sides are taken over multi-indices of possibly 
different dimensions, relevant to different sets of variables.
As a special case it includes  a multiple summation formula of type $A_{n}$ which generalizes the Frenkel\,--\,Turaev summation.
This duality transformation formula can also be interpreted as  a kernel identity for the commuting family of
Ruijsenaars difference operators (\cite{KoNS1}).
We refer the reader to Kajihara \cite{Ka1} for
the duality transformation formula of type $A_n$
for basic hypergeometric series and its relation to Macdonald $q$-difference operators.

In this paper we investigate duality transformation formulas for multiple elliptic hypergeometric series of type $BC_n$ and of type $C_n$.
We formulate two kinds of duality transformations, one on subsets and the other on multi-indices.
In our approach we first establish
duality transformations on subsets combining some kernel identities for Ruijsenaars\,--\,van Diejen difference operators
in one variable and the elliptic Cauchy determinant formula.
Duality transformations on multi-indices are then constructed
from those on subsets through the method of multiple principal specialization as proposed in \cite{KaN1}.
We also discuss some summation formulas which are obtained from our duality transformations.
It would be an important problem to clarify the relation between our results for multiple series
and summation and transformation formulas for elliptic hypergeometric integrals of type $BC_n$ as discussed by
van Diejen\,--\,Spiridonov \cite{vDS1} and Rains \cite{Ra1}. 
In a previous work, 
the second author \cite{M1} derived a duality transformation formula for basic hypergeometric series of type $BC_n$
from a kernel identity for the commuting family of van Diejen $q$-difference operators.
We expect that our transformation formulas can also be understood in a context similar to that of \cite{M1}.
\par\medskip
The present paper is organized as follows. 
We begin by proposing in Section 2 kernel identities for 
Ruijsenaars\,--\,van Diejen difference operators in one variable 
$L(x;\bs{a}|\bs{c})$  of type $BC_1$ and $R(x;\bs{a})$ of type $C_1$.  
In this paper we call a difference operator in $n$ variables 
{\em of type $BC_n$} (resp. {\em of type $C_n$})
if it is invariant under the action of the hyperoctahedral group of degree $n$, 
and contains {\em eight} (resp. {\em four})
parameters 
$\bs{a}=(a_0,a_1,\ldots)$ 
possibly with some constraints. 
In Section 3 we construct a duality transformation formula 
{\em on subsets\/} (Theorem \ref{Theo:BC set}) 
of type $BC_n$ 
applying the product of one variable difference 
operators 
to an elliptic version 
of the Cauchy determinant.  We also derive a 
duality transformation on subsets (Theorem \ref{Theo:C set1})
of type $C_n$ from the $BC_n$ case 
by a specialization of parameters. 

We apply the method 
of multiple principal specialization as in Kajihara\,--\,Noumi \cite{KaN1}
to the transformation formulas on subsets 
for constructing duality transformation formulas {\em on multi-indices}.  
Section 4 is devoted to the study of the $(C_m,C_n)$ case (Theorem \ref{Theo:C dual}); it contains 
some remarks on special cases for comparison 
with previously known results,
including the summation formula of type $C_n$ by Rosengren \cite{Ro2}.  
As an application of Theorem \ref{Theo:C dual}
we also derive a transformation formula 
of Karlsson\,--\,Minton type (Theorem \ref{Theo:transformation N+r to N+s}), 
which gives a multiple generalization of 
the Karlsson\,--\,Minton type transformation due to Rosengren\,--\,Schlosser \cite{RS1}.  
In Section 5 we propose a duality transformation of type 
$(BC_m,BC_n)$ (Theorem \ref{Theo:BC multi1}).   With an additional constraint on the 
parameters, this formula is further generalized to a duality transformation
intertwining multiple series of different degrees. 

\section{Kernel identities in one variable}
In this section, we introduce a Ruijsenaars\,--\,van Diejen type 
difference operator $L(x; \bs{a} |\bs{c})$ of type $BC_1$ with
complex
parameters $\bs{a} =(a_{\pa}, a_{\pb}, \ldots, a_{\ph})$ and 
$\bs{c}=(c_0, c_1, c_2, c_3)$,
and prove a kernel identity for this operator.
We also derive a kernel identity for a difference operator $R(x; \bs{a})$ of type $C_1$
with four parameters $\bs{a} =(a_{\pa}, a_{\pb}, a_\pc, a_{\pd})$.
\subsection{Ruijsenaars\,--\,van Diejen operators and kernel identities}
Let $[u]$ be a non-zero entire 
odd function in one variable $u$, satisfying the functional equation
\begin{equation}\label{eq:Riemann}
[x \pm u][y \pm  v]-[x \pm v] [y \pm u] =[x \pm y] [u \pm v],
\end{equation} 
where we used an abbreviated notation $[x \pm y]:=[x+y][x-y]$.
We denote by $\Omega$ the set of all zeros of $[u]$.
Then $\Omega$ is a closed discrete subgroup of $\mathbb{C}$.
It is known that such functions are classified, up to constant multiples, into the following three types
by the rank of $\Omega \subset \mathbb{C}$:
\begin{alignat}{4}
& (0)\ \mbox{rational case:} \qquad && e\! \left(a u^2\right)u \qquad &&(\Omega=0),& \nonumber \\ 
& (1)\ \mbox{trigonometric case:} \qquad& & e\! \left(a u^2\right)\sin(\pi u/\omega_1)\qquad && (\Omega=\mathbb{Z}\omega_1),& \nonumber \\
& (2)\ \mbox{elliptic case:}  \qquad && e\! \left(au^2\right)\sigma(u;\Omega)\qquad &&(\Omega=\mathbb{Z}\omega_1\oplus\mathbb{Z}\omega_2),& \label{eq:three cases}
\end{alignat}
where $a \in \mathbb{C}$ and $e(u) = \exp(2\pi \sqrt{-1} u)$.
We denote by $\sigma(u; \Omega)$
the Weierstrass sigma function 
\begin{equation*}
\sigma(u; \Omega) = u \prod_{\omega \in \Omega\backslash\{0\}}
\left(1 - \frac{u}{\omega} \right)
e^{\frac{u}{\omega}+\frac{u^2}{2\omega^2}}
\end{equation*}
associated with the period lattice
$\Omega=\mathbb{Z}\omega_1\oplus\mathbb{Z}\omega_2$, where $\omega_1$ and $\omega_2$ are two $\mathbb{R}$-linearly independent 
complex numbers.
In the following, we confine ourselves to the elliptic case 
$[u] =e\left(au^2\right)\sigma(u;\Omega)$.
In this setting, the function $[u]$ has 
the following quasi-periodicity:
\begin{equation*}
[u + \omega] =
\epsilon_{\omega} e\left(\eta_{\omega} \left(u+\frac{\omega}{2} \right) \right)
[u] \quad (\omega \in \Omega),
\end{equation*}
where 
 $\eta_{\omega} \in \mathbb{C}\ (\omega \in \Omega)$
 and
\begin{equation*}
\epsilon_{\omega} = 
\begin{cases} 
1  &(\omega \in 2\Omega), \\
-1 &(\omega \notin 2\Omega).
\end{cases}
\end{equation*}
We also denote $\omega_3 = -\omega_1-\omega_2, \omega_0 = \omega_4=0$
and set 
$\epsilon_r = \epsilon_{\omega_r}, \eta_r = \eta_{\omega_r}$
for $r=0, 1, 2, 3$. 
In this notation,
the duplication formula for $[u]$ is given by
\begin{equation}\label{eq:double-angle formula}
[2x] = 2[x] \prod_{s=1}^3 
\frac{[x-\frac{1}{2}\omega_s]}
{[-\frac{1}{2}\omega_s]}.
\end{equation}
We remark that the following formula follows from \eqref{eq:double-angle formula}:
\begin{equation}\label{eq:sp of double-angle formula}
\prod_{0 \le s \le 3; s \neq r}[\tfrac{1}{2}(\omega_r-\omega_s) ]
 = \epsilon_r e(\tfrac{1}{2}\eta_r \omega_r )
 \prod_{s=1}^3 [-\tfrac{1}{2}\omega_s ] 
\quad (r=0, 1, 2, 3).
\end{equation}
Fixing a complex parameter $\delta \in \mathbb{C}^*$ such that $\mathbb{Z}\delta \cap \Omega = \{ 0 \}$,
we define the difference operator $L(x;\bs{a}|\bs{c})$ as follows:
\begin{align*}
L(x;\bs{a}|\bs{c}) &= A^+(x; \bs{a})T_x^{\delta} + A^-(x; \bs{a})T_x^{-\delta}+A^0(x; \bs{a}|\bs{c}) ,\\
A^+(x; \bs{a})&= \frac{\prod_{p=\pa}^\ph[x+a_p]}{[2x][2x+\delta]} ,\quad
A^-(x; \bs{a})= \frac{\prod_{p=\pa}^\ph[x-a_p]}{[2x][2x-\delta]} ,\\
A_r^0(x;\bs{a}|c) &= 
\epsilon_r e((\delta - \tfrac{\omega_r}{2} - \tfrac{1}{2} {\textstyle \sum_{p=\pa}^\ph} a_p)\eta_r)
\frac{[x \pm c] \prod_{p=\pa}^\ph[\frac{1}{2}(\omega_r-\delta)+a_p]
}
{2[\frac{1}{2}(\omega_r-\delta) \pm x][\frac{1}{2}(\omega_r-\delta) \pm c]},\\
A^0(x;\bs{a}|\bs{c})&=\sum_{r=0}^3 A^0_r(x;\bs{a}|c_r),
\end{align*}
where $T_{x}^\delta$ stands for the $\delta$-shift operator
$T_x^{\delta} f(x) = f(x+ \delta)$.
Note that this operator is invariant under the sign change $x \rightarrow -x$.
This operator $L(x;\bs{a}|\bs{c})$ is essentially equivalent to the $BC_1$ Ruijsenaars\,--\,van Diejen operator defined by \cite{vD1, KH, Ru2}
(See also \cite{KoNS1}).
We remark that the parameters $\bs{c}=(c_0, c_1, c_2, c_3)$ in $A^0(x; \bs{a}| \bs{c})$ are superfluous.
In fact, from 
\begin{equation*}
\frac{1}{[\frac{1}{2}(\omega_r-\delta) \pm x]}
\left( \frac{[x\pm c]}{[\frac{1}{2}(\omega_r-\delta) \pm c]}
- \frac{[x\pm c']}{[\frac{1}{2}(\omega_r-\delta) \pm c']} \right)
= \frac{[c' \pm c]}{[\frac{1}{2}(\omega_r-\delta) \pm c] [\frac{1}{2}(\omega_r-\delta) \pm c']},
\end{equation*}
we obtain
\begin{equation*}
A^0_r(x; \bs{a}|c)- A^0_r(x; \bs{a}|c')= A^0_r(c'; \bs{a}| c).
\end{equation*}
The difference operator $(2. 13)$ in \cite{KoNS1} with $m=1$
is a special case  $L(x; \bs{a}| \bs{c'})$ where $c'_r = \frac{1}{2}(\omega_r-\delta)+\kappa\ (r=0, 1, 2, 3)$,
and hence, for arbitrary $\bs{c}=(c_0, c_1, c_2, c_3)$  it differs from $L(x; \bs{a}| \bs{c})$
by an additive constant.

For the difference operator $L(x;\bs{a}|\bs{c})$,
the following kernel identity holds.
\begin{theo}\label{Theo:BC_1}
Define the parameters $\bs{b} =(b_{\pa}, b_{\pb}, \ldots, b_{\ph} )$ by $b_{p} = \delta - a_{p}\ (p=\pa, \pb, \ldots, \ph )$.
\\
\noindent
{\rm (1)}\ Under the balancing condition $\sum_{p=\pa}^{\ph}a_p=4 \delta$, the following identity holds\,$:$
\begin{equation}\label{eq:KI m=n=1}
L(x;\bs{a} | \bs{c}) \left(\frac{1}{[x \pm y]} \right) 
= L(y; \bs{b} | \bs{c}) \left(\frac{1}{[x \pm y]} \right) .
\end{equation}
{\rm(2)}\ Under the balancing condition $\sum_{p=\pa}^{\ph} a_p =2\delta$,
the following identity holds\,$:$
\begin{equation}\label{eq:KI m=1 n=0}
L(x;\bs{a}| \bs{a_\pa}) \cdot 1 = \frac{\prod_{p=\pb}^\ph[a_\pa+a_p] }{[2a_\pa +\delta]},
\end{equation}
where $\bs{a_\pa}=(a_\pa, a_\pa, a_\pa, a_\pa)$.
\end{theo}
We will give a proof of Theorem \ref{Theo:BC_1} in the next subsection.
We now consider the special case 
$\{a_\pe, a_\pf, a_\pg, a_\ph \} = \{-\frac{1}{2}(\omega_r-\delta)\  (r=0, 1, 2, 3)\}$.
Then, from the duplication formula
\eqref{eq:double-angle formula}
we obtain
\begin{equation*}
\prod_{p=\pe}^\ph [x+a_p] 
= \prod_{r=0}^3 [x+\tfrac{\delta}{2} - \tfrac{\omega_r}{2}]
=\frac{[2x+\delta] \prod_{s=1}^3 [-\frac{1}{2}\omega_s]}{2}
\end{equation*}
and
\begin{equation*}
\prod_{p=\pe}^\ph [y+b_p] 
= \prod_{r=0}^3 [-y-\tfrac{\delta}{2} - \tfrac{\omega_r}{2}]
=\frac{[-2y-\delta] \prod_{s=1}^3 [-\frac{1}{2}\omega_s]}{2}.
\end{equation*}
We also note that $A^0(x; \bs{a} | \bs{c}) = A^0(y;\bs{b} | \bs{c})=0$
and $\sum_{p=\pa}^\ph a_p = \sum_{p=\pa}^\pd a_p + 2\delta$.
By this specialization we obtain the kernel identity for a difference operator of type $C$ from Theorem \ref{Theo:BC_1}
after replacing $\delta$ by $\frac{\delta}{2}$.
 
We define the difference operator
$R(x; \bs{a})$ of type $C_1$ with four parameters $\bs{a} = (a_\pa, a_\pb, a_\pc, a_\pd)$ as follows:
\begin{align*}
R(x; \bs{a}) &= B^+(x;\bs{a})T_x^{\delta/2} + B^-(x;\bs{a})T_x^{-\delta/2},\\
B^+(x;\bs{a})&= \frac{\prod_{p=\pa}^\pd[x+a_p]}{[2x]} ,\quad
B^-(x;\bs{a})= -\frac{\prod_{p=\pa}^\pd[x-a_p]}{[2x]}.
\end{align*}
\begin{theo}\label{Theo:C_1}
Define the parameters $\bs{b} =(b_{\pa}, b_{\pb}, b_{\pc}, b_{\pd})$ by $b_{p} = \delta/2 - a_{p}\ 
(p=\pa, \pb, \pc, \pd)$.
\\
\noindent
{\rm (1)}\ Under the balancing condition $\sum_{p=\pa}^\pd{a_p}= \delta$, the following identity holds\,$:$
\begin{equation*} 
R(x; \bs{a}) \left(\frac{1}{[x \pm y]} \right) 
= R(y;\bs{b}) \left(\frac{1}{[y \pm x]} \right) .
\end{equation*}
{\rm(2)}\ Under the balancing condition $\sum_{p=\pa}^\pd a_p=0$,
the following identity holds\,$:$
\begin{equation*}
R(x; \bs{a}) \cdot 1 = \prod_{p=\pb}^\pd[a_\pa+a_p].
\end{equation*}
\end{theo}
\subsection{Proof of the kernel identity}
We first recall the formula for partial fraction decomposition (see \cite[(11.7.6), p.~332]{GR}).
\begin{prop}\label{prop:partial fraction}
For variables $z, (x_1, \ldots, x_N)$ and $(y_1, \ldots, y_N)$,
we have
\begin{equation*} 
[c] \prod_{j=1}^N\frac{[z-y_j]}{[z-x_j]}
=\sum_{i=1}^N \frac{[z-x_i+c]}{[z-x_i]}
\frac{\prod_{1 \le j \le N}[x_i-y_j]}
{\prod_{1 \le j \le N; j \neq i}[x_i-x_j]},
\end{equation*}
where $c= \sum_{i=1}^N (x_i-y_i)$. 
\end{prop}
Applying this formula,
we decompose the function
\begin{equation*}
F(z)= \frac{\prod_{p=0}^{m+3}[z+a_p-\tfrac{1}{2}\delta]}
{[z \pm x-\tfrac{1}{2}\delta][z \pm y +\tfrac{1}{2}\delta]
\prod_{r=0}^{m-1}[z+ d_r-\tfrac{1}{2}\delta]}
\end{equation*}
into partial fractions
with $c= \sum_{p=0}^{m+3} a_p-\sum_{r=0}^{m-1} d_r-2\delta$.
Then by setting $c=0$
we obtain the following lemma.
\begin{lemm}\label{lemm:general identity}
For two sets of parameters $\bs{a} = (a_{\pa}, a_{\pb}, \ldots, a_{m+\pd})$
and $\bs{d}=(d_0, d_1, \ldots, d_{m-1})$ satisfying 
$\sum_{p=\pa}^{m+\pd} a_p- \sum_{r=0}^{m-1} d_r=2\delta$ and 
$d_r \not\equiv d_{r'} \pmod{\Omega}\ (0 \le r, r' \le m-1; r \neq r')$,
we have the following identity\,$:$
\begin{multline}\label{eq:general identity}
\frac{\prod_{p=\pa}^{m+\pd} [x+a_p]}
{[2x] \prod_{r=0}^{m-1} [x+d_r]}
\frac{[x \pm y]}{[(x+\delta) \pm y]}
+\frac{\prod_{p=\pa}^{m+\pd} [-x+a_p]}
{[-2x] \prod_{r=0}^{m-1} [-x+d_r]}
\frac{[x \pm y]}{[(x-\delta) \pm y]} \\
+\frac{\prod_{p=\pa}^{m+\pd} [y+b_p]}
{[2y] \prod_{r=0}^{m-1} [y+e_r]}
\frac{[x \pm y]}{[x \pm (y+\delta)]}
+\frac{\prod_{p=\pa}^{m+\pd} [-y+b_p]}
{[-2y] \prod_{r=0}^{m-1} [-y+e_r]}
\frac{[x \pm y]}{[x \pm (y -\delta)]} \\
+\sum_{r=0}^{m-1} 
\frac{\prod_{p=\pa}^{m+\pd}[a_p-d_r]}
{\prod_{r'=0;r'\neq r}^{m-1} [d_{r'}-d_r]}
\frac{[x \pm y]}{[-d_r \pm x] [-e_r \pm y]} =0,
\end{multline}
where $b_p = \delta- a_p\ (p=\pa, \pb, \ldots, m+3)$ and 
$e_r = \delta- d_r\ (r=0, 1, \ldots, m-1)$.
\end{lemm}
Note that this formula is invariant under the sign changes $x \rightarrow -x$ and $y \rightarrow -y$.
\par\medskip
We remark that a function of the form
$\frac{[x \pm y]}{[a \pm x] [a' \pm y]}$ with $a \equiv a' \pmod{\Omega}$
can be separated into a sum of two functions,
one depending only on $x$ and the other only on $y$.
In fact, for $\omega \in \Omega$ we have
\begin{align*}
\frac{[x \pm y]}{[a \pm x] [a+\omega \pm y]} &= 
e(-\eta_\omega (2a+\omega) ) \frac{[x \pm y]}{[a \pm x] [a \pm y]} \nonumber \\
&=
e(-\eta_\omega (2a+\omega) ) \left(\frac{[x \pm b]}{[a \pm x] [a \pm b]}
-\frac{[y \pm b]}{[a \pm y] [a \pm b]} \right) .  
\end{align*}
We assume $d_r \equiv e_r \pmod{\Omega}$
in \eqref{eq:general identity}.
Then from $e_r = \delta- d_r\ (r=0, 1, \ldots, m-1)$
we have $2d_r \equiv \delta \pmod{\Omega}$. 
In view of the condition $d_r \not\equiv d_{r'} \pmod{\Omega}\ (0 \le r, r' \le m-1; r \neq r')$
we find that the terms involving the factor 
$\frac{[x \pm y]}{[-d_r \pm x] [-e_r \pm y]}$
in \eqref{eq:general identity} can be separated  
when $m = 4$ and $d_r=\frac{1}{2}(\delta-\omega_r)\ (r=0, 1, 2, 3)$.

We prove Theorem \ref{Theo:BC_1} $(1)$ by using Lemma \ref{lemm:general identity}
for $m=4$ and 
$d_r = \frac{1}{2}(\delta- \omega_r)\ (r=0, 1, 2, 3)$.
From $\sum_{p=\pa}^\ph a_p -\sum_{r=0}^3 d_r = 2 \delta$,
we obtain $\sum_{p=\pa}^\ph a_p = 4\delta$.
We compute each term of \eqref{eq:general identity} in this setting.
From formula \eqref{eq:double-angle formula},
we have
\begin{align*}
\frac{\prod_{p=\pa}^{\ph} [x+a_p]}
{[2x] \prod_{r=0}^3 [x+d_r]}
\frac{[x \pm y]}{[(x+\delta) \pm y]}
&= \frac{\prod_{p=\pa}^{\ph} [x+a_p]}
{[2x] \prod_{r=0}^3 [x+\frac{1}{2}(\delta-\omega_r)]}
\frac{[x \pm y]}{[(x+\delta) \pm y]} \nonumber \\
&=
\frac{2}{\prod_{s=1}^3 [-\frac{1}{2}\omega_s ]}
\frac{\prod_{p=\pa}^{\ph} [x+a_p]}
{[2x] [2x+\delta]}
\frac{[x \pm y]}{[(x+\delta) \pm y]}, \\
\frac{\prod_{p=\pa}^{\ph} [y+b_p]}
{[2y] \prod_{r=0}^3 [y+e_r]}
\frac{[x \pm y]}{[x \pm (y+\delta)]}
&=\frac{\prod_{p=\pa}^{\ph} [y+b_p]}
{[2y] \prod_{r=0}^3 [y+\frac{1}{2}(\delta + \omega_r) ]}
\frac{[x \pm y]}{[x \pm (y+\delta)]} \nonumber \\
&=-\frac{2}{\prod_{s=1}^3 [-\frac{1}{2}\omega_s ]}
\frac{\prod_{p=\pa}^{\ph} [y+b_p]}
{[2y] [2y+\delta]}
\frac{[x \pm y]}{[x \pm (y+\delta)]}.
\end{align*}
By using \eqref{eq:sp of double-angle formula}
the last term in \eqref{eq:general identity} is transformed into
\begin{align*}
&\sum_{r=0}^3 
\frac{\prod_{p=\pa}^{\ph}[a_p-d_r]}
{\prod_{r'=0;r'\neq r}^3 [d_{r'}-d_r]}
\frac{[x \pm y]}{[-d_r \pm x] [-e_r \pm y]} \nonumber\\
&=
\sum_{r=0}^3 
\frac{\prod_{p=\pa}^{\ph}[a_p-\frac{1}{2}(\delta-\omega_r) ]}
{\prod_{r'=0;r'\neq r}^3 [\frac{1}{2}(\omega_r- \omega_{r'})]}
\frac{[x \pm y]}
{[\frac{1}{2}(\omega_r-\delta)\pm x] [\frac{1}{2}(-\delta-\omega_r) \pm y]} \nonumber\\
&=
\sum_{r=0}^3 
\epsilon_r e( (-\delta -\tfrac{\omega_r}{2}) \eta_r)
\frac{ \prod_{p=\pa}^\ph[\frac{1}{2}(\omega_r-\delta)+a_p]}
{\prod_{s=1}^3[-\frac{1}{2}\omega_s]}
\frac{[x \pm c_r]}
{[\frac{1}{2}(\omega_r-\delta) \pm x][\frac{1}{2}(\omega_r-\delta) \pm c_r]} \nonumber \\
&-\sum_{r=0}^3 
\epsilon_r e( (-\delta -\tfrac{\omega_r}{2}) \eta_r)
\frac{ \prod_{p=\pa}^\ph[\frac{1}{2}(\omega_r-\delta)+b_p]}
{\prod_{s=1}^3[-\frac{1}{2}\omega_s]}
\frac{[y \pm c_r]}
{[\frac{1}{2}(\omega_r-\delta) \pm y][\frac{1}{2}(\omega_r-\delta) \pm c_r]}.
\end{align*}
This completes the proof of the identity \eqref{eq:KI m=n=1}.
\par\medskip
We next prove Theorem \ref{Theo:BC_1} $(2)$  by specializing
the variable $y$ and the parameters $\bs{c}=(c_0, c_1, c_2, c_3)$ in \eqref{eq:KI m=n=1} as 
$y=b_\pa$
and $\bs{c}=\bs{b_\pa}=(b_\pa, b_\pa, b_\pa, b_\pa)$.
Then from $A^-(y;\bs{b}) =A^0(y;\bs{b} |\bs{c}) =0$,
we have
\begin{equation}\label{eq:sub of Theo 2.1 (2)}
A^+(x;\bs{a}) \frac{[x \pm (a_\pa-2\delta)]}{[x+\delta \pm (\delta-a_\pa)]}
+A^-(x;\bs{a}) \frac{[x \pm (a_\pa-2\delta)]}{[x-\delta \pm (\delta-a_\pa)]}+
A^0(x;\bs{a}|\bs{b_\pa}) \frac{[x \pm (a_\pa-2\delta)]}{[x \pm (\delta-a_\pa)]}
= A^+(b_\pa; \bs{b}).
\end{equation}
If we replace $a_0$ by $a_0+2\delta$, 
after a non-trivial recombination process of factors 
this formula is translated into the equality
\begin{equation*}
A^+(x;\bs{a}) +A^-(x;\bs{a})+ 
A^0(x;\bs{a}| \bs{a_\pa}) = \frac{\prod_{p=\pb}^\ph[a_\pa+a_p]}{[2a_\pa+\delta]}
\end{equation*}
under the balancing condition $\sum_{p=0}^7 a_p =2\delta$.
For example, the third term of \eqref{eq:sub of Theo 2.1 (2)} is computed as
\begin{align*}
&A^0_r(x;a_0+2\delta, a_1, \ldots, a_7| -\delta-a_0) \frac{[x \pm a_0]}{[x \pm (-\delta-a_0)]}\nonumber \\
&= \epsilon_r e((\delta - \tfrac{\omega_r}{2} - \tfrac{1}{2} {\textstyle \sum_{p=\pa}^\ph} a_p-\delta)\eta_r)
\frac{[x \pm a_0] [\frac{1}{2}(\omega_r-\delta)+a_{\pa}+2\delta]
\prod_{p=\pb}^\ph[\frac{1}{2}(\omega_r-\delta)+a_p]
}
{2[\frac{1}{2}(\omega_r-\delta) \pm x][\frac{1}{2}(\omega_r-\delta) \pm (-\delta-a_0)]} \nonumber\\
&=
\epsilon_r e((\delta - \tfrac{\omega_r}{2} - \tfrac{1}{2} {\textstyle \sum_{p=\pa}^\ph} a_p-\delta)\eta_r)
\frac{[x \pm a_0] 
[\frac{1}{2}\omega_r + \frac{3}{2}\delta+a_{\pa}]
\prod_{p=\pb}^\ph[\frac{1}{2}(\omega_r-\delta)+a_p]
}
{2[\frac{1}{2}(\omega_r-\delta) \pm x] [\frac{1}{2}\omega_r-\frac{3}{2}\delta -a_0]
[\frac{1}{2}\omega_r+\frac{1}{2}\delta +a_0]} \nonumber \\
&=
\epsilon_r e((\delta - \tfrac{\omega_r}{2} - \tfrac{1}{2} {\textstyle \sum_{p=\pa}^\ph} a_p)\eta_r)
\frac{[x \pm a_0] 
\prod_{p=\pb}^\ph[\frac{1}{2}(\omega_r-\delta)+a_p]
}
{2[\frac{1}{2}(\omega_r-\delta) \pm x]
[\frac{1}{2}(\omega_r-\delta) -a_0]} \nonumber\\
&= A_r^0(x;\bs{a}|a_0) .
\end{align*}
This completes the proof of \eqref{eq:KI m=1 n=0}.

\section{Duality transformation formulas on subsets}
In this section, we derive duality transformation formulas of type $BC$ and $C$ 
by combining Theorem \ref{Theo:BC_1} (1) and the Cauchy determinant formula for the function $[u]$
in the spirit of \cite{KaN1}.
\subsection{Duality transformation of type $BC$ on subsets}
For the variables $z=(z_1, \ldots, z_N)$ and $w=(w_1, \ldots, w_N)$,
we investigate an elliptic version of the Cauchy determinant
\begin{equation*}
D(z|w):= \det \left( \frac{1}{[z_i \pm w_j]} \right)_{i, j=1}^N.
\end{equation*}
This determinant is factorized as follows:
\begin{equation}\label{eq:determinant formula}
D(z|w) = (-1)^{\binom{N}{2}}
\frac{\prod_{1 \le i < j \le N}[z_i \pm z_j][w_i \pm w_j ]}
{\prod_{1 \le i , j \le N} [z_i \pm w_j]}.
\end{equation}
Formula \eqref{eq:determinant formula} is proved for example 
by using the functional relation \eqref{eq:Riemann} and Jacobi's identity for determinants.
It is a variant of the elliptic Cauchy determinant due to Frobenius \cite{F}.

By applying the difference operator
$E(z;\bs{a} | \bs{c}):=\prod_{i=1}^N (u+ L(z_i;\bs{a} | \bs{c}))$
with a parameter $u$ to
the determinant $D(z|w)$,
from Theorem \ref{Theo:BC_1} (1) we have 
\begin{align*}
E(z;\bs{a} | \bs{c}) D(z|w) 
&= \det \left( \frac{u}{[z_i \pm w_j]} + L(z_i;\bs{a} | \bs{c}) \frac{1}{[z_i \pm w_j]} \right)_{i, j=1}^N \nonumber \\
&= \det \left( \frac{u}{[z_i \pm w_j]} + L(w_j;\bs{b} | \bs{c}) \frac{1}{[z_i \pm w_j]} \right)_{i, j=1}^N \nonumber \\
&= E(w;\bs{b} | \bs{c}) D(z|w) 
\end{align*}
under the balancing condition $\sum_{p=\pa}^\ph a_p = 4\delta$
and $b_p= \delta -a_p\ (p=\pa, \pb, \ldots, \ph)$.
Since $D(z|w)=(-1)^{N^2} D(w|z)$,
by dividing both sides by $D(z|w)$
we have
\begin{equation*}
\frac{E(z;\bs{a} | \bs{c}) D(z|w)}
{D(z|w)} 
= \frac{E(w;\bs{b} | \bs{c}) D(w|z)}{D(w|z)} .
\end{equation*}
We expand the operator $E(z;\bs{a} | \bs{c})$ as
\begin{equation*}
E(z;\bs{a} | \bs{c})
= \sum_{r=0}^N u^{N-r}
\sum_{\substack{I \subset \{1, \ldots, N \}\\ |I|=r}}
\sum_{I_+ \cup I_0 \cup I_- = I}
\prod_{i \in I_+} A^+(z_i;\bs{a}) \prod_{i \in I_0} A^0(z_i;\bs{a} | \bs{c}) 
\prod_{i \in I_-} A^-(z_i;\bs{a})
\prod_{i \in I_+} T_{z_i}^{\delta}
\prod_{i \in I_-} T_{z_i}^{-\delta}.
\end{equation*}
Here the third summation is taken over the all triples $(I_+, I_0, I_-)$ of subsets of $I$ such that
\begin{equation*}
I_+ \cup I_0 \cup I_- = I, \quad |I_+|+|I_0|+|I_-|=|I|.
\end{equation*}
If we set $\epsilon_i= +, 0, -$ according as $i$ belongs to $I_+, I_0, I_- (i=1, \ldots, N)$,
$E(z;\bs{a} | \bs{c})$ is alternatively  expressed as
\begin{equation*}
E(z;\bs{a} | \bs{c})
= \sum_{r=0}^N u^{N-r}
\sum_{\substack{I \subset \{1, \ldots, N \}\\ |I|=r }}
\sum_{\epsilon \in \{ \pm, 0 \}^I}
\prod_{i \in I} A^{\epsilon_i} (z_i;\bs{a}) 
\prod_{i \in I} T_{z_i}^{\epsilon_i \delta},
\end{equation*}
where we have omitted the parameters $\bs{c}$ in $A^0(z; \bs{a}|\bs{c})$.
Applying the shift operator 
$\prod_{i=1}^N T_{z_i}^{\epsilon_i\delta}\ 
(\epsilon_i \in \{\pm , 0 \})$ to $D(z|w)$,
we have
\begin{equation*}
\frac{\prod_{i=1}^N T_{z_i}^{\epsilon_i\delta}D(z|w)}
{D(z|w)}
=\prod_{1 \le i < j \le N}
\frac{[(z_i+\epsilon_i \delta) \pm (z_j + \epsilon_j \delta)]}
  {[z_i \pm z_j]}
\prod_{\substack{1 \le i \le N \\1 \le k \le N}}
\frac{[z_i \pm w_k] }{[(z_i + \epsilon_i \delta) \pm w_k ]} .
\end{equation*}
Hence we obtain
\begin{align*}
\frac{E(z;\bs{a} | \bs{c})D(z|w)}
{D(z|w)}
=&
\sum_{r=0}^N u^{N-r}
\sum_{\substack{I \subset \{1, \ldots, N \}\\ |I|=r }}
\sum_{\epsilon \in \{ \pm, 0 \}^I}
\prod_{i \in I} A^{\epsilon_i} (z_i;\bs{a}) \nonumber\\
&\cdot
\prod_{1 \le i < j \le N}
\frac{[(z_i+\epsilon_i \delta) \pm (z_j + \epsilon_j \delta)]}
  {[z_i \pm z_j]}
\prod_{\substack{1 \le i \le N \\1 \le k \le N}}
\frac{[z_i \pm w_k] }{[(z_i + \epsilon_i \delta) \pm w_k ]} \nonumber\\
=&
\sum_{r=0}^N u^{N-r}
\sum_{\substack{I \subset \{1, \ldots, N \}\\ |I|=r}}
\sum_{I_+ \cup I_- \cup I_0 = I}
\prod_{i \in I_+} A^+(z_i;\bs{a}) \prod_{i \in I_0} A^0(z_i;\bs{a} | \bs{c})
\prod_{i \in I_-} A^-(z_i; \bs{a})
\nonumber\\
&\cdot
\prod_{\{i, j \} \subset I_+} \frac{[z_i +z_j +2\delta]}{[z_i +z_j]}
\prod_{\{i, j \} \subset I_-} \frac{[z_i +z_j -2\delta]}{[z_i +z_j]}
\prod_{\substack{i \in I_+ \\ j \in I_-} }\frac{[z_i - z_j +2\delta]}{[z_i - z_j]} \nonumber \\
&\cdot
\prod_{\substack{i \in I_+ \\ j \in I_0} }
\frac{[z_i + \delta \pm z_j]}{[z_i \pm z_j]}
\prod_{\substack{i \in I_- \\ j \in I_0} }
\frac{[z_i - \delta \pm z_j]}{[z_i \pm z_j]}
\prod_{\substack{i \in I_+ \\ 1 \le k \le N}}
\frac{[z_i \pm w_k]}{[z_i + \delta \pm w_k]}
\prod_{\substack{i \in I_- \\ 1 \le k \le N}}
\frac{[z_i \pm w_k]}{[z_i - \delta \pm w_k]}.
\end{align*}
Here $\prod_{\{i, j \} \subset I_{\epsilon}}\ (\epsilon=\pm, 0)$ stands for
the product  over all two-element subsets of $I_{\epsilon}$.
By computing $\frac{E(w;\bs{b} | \bs{c}) D(w|z)}{D(w|z)}$,
we obtain the following theorem.
\begin{theo}\label{Theo:BC set}
For any complex parameters $\bs{a} = (a_\pa, a_\pb, \ldots, a_\ph)$,
we define the parameters $\bs{b} =(b_\pa, b_\pb, \ldots, b_\ph)$ by $b_p = \delta - a_p\ (p=\pa, \pb, \ldots, \ph)$.
Suppose that the balancing condition $\sum_{p=\pa}^\ph a_p = \sum_{p=\pa}^\ph b_p = 4\delta$ 
is satisfied. Then, 
for two sets of variables $z=(z_1, \ldots, z_N)$ and $w=(w_1, \ldots, w_N)$,
the following identity holds for $r=0, 1, \ldots, N$\,$:$
\begin{align}
&\sum_{\substack{I \subset \{1, \ldots, N \}\\ |I|=r}}
\sum_{I_+ \cup I_- \cup I_0 = I}
\prod_{i \in I_+} A^+(z_i; \bs{a}) \prod_{i \in I_0} A^0(z_i;\bs{a} | \bs{c})
\prod_{i \in I_-} A^-(z_i;\bs{a}) \nonumber\\
&\cdot
\prod_{\{i, j \} \subset I_+} \frac{[z_i +z_j +2\delta]}{[z_i +z_j]}
\prod_{\{i, j \} \subset I_-} \frac{[z_i +z_j -2\delta]}{[z_i +z_j]}
\prod_{\substack{i \in I_+ \\ j \in I_-} }\frac{[z_i - z_j +2\delta]}{[z_i - z_j]} \nonumber \\
&\cdot
\prod_{\substack{i \in I_+ \\ j \in I_0} }
\frac{[z_i + \delta \pm z_j]}{[z_i \pm z_j]}
\prod_{\substack{i \in I_- \\ j \in I_0} }
\frac{[z_i - \delta \pm z_j]}{[z_i \pm z_j]}
\prod_{\substack{i \in I_+ \\ 1 \le k \le N}}
\frac{[z_i \pm w_k]}{[z_i + \delta \pm w_k]}
\prod_{\substack{i \in I_- \\ 1 \le k \le N}}
\frac{[z_i \pm w_k]}{[z_i - \delta \pm w_k]}\nonumber \\
&=\sum_{\substack{K \subset \{1, \ldots, N \}\\ |K|=r}}
\sum_{K_+ \cup K_- \cup K_0 = K}
\prod_{k \in K_+} A^+(w_k;\bs{b}) \prod_{k \in K_0} A^0(w_k;\bs{b} | \bs{c})
\prod_{k \in K_-} A^-(w_k; \bs{b})\nonumber\\
&\cdot
\prod_{\{k, l \} \subset K_+} \frac{[w_k +w_l +2\delta]}{[w_k+w_l]} 
\prod_{\{k, l \} \subset K_-} \frac{[w_k+w_l- 2\delta]}{[w_k+w_l]}
\prod_{\substack{k \in K_+ \\ l \in K_-}}\frac{[w_k-w_l +2\delta]}{[w_k-w_l]}\nonumber\\
&\cdot
\prod_{\substack{k \in K_+ \\ l \in K_0}}
\frac{[w_k+\delta \pm w_l ]}{[w_k \pm w_l]}
\prod_{\substack{k \in K_- \\ l \in K_0}}
\frac{[w_k-\delta \pm w_l ]}{[w_k \pm w_l]} 
\prod_{\substack{k \in K_+ \\ 1 \le i \le N}}
\frac{[w_k \pm z_i]}{[w_k +\delta \pm z_i]}
\prod_{\substack{k \in K_- \\ 1 \le i \le N}}
\frac{[w_k \pm z_i]}{[w_k -\delta \pm z_i]} \label{eq:BC set1}.
\end{align}
\end{theo}
This formula is invariant under the action of the Weyl group of type $C$ on
$z$ variables and $w$ variables, respectively.
\subsection{Duality transformation of type $C$ on subsets}
We specialize the parameters in Theorem \ref{Theo:BC set} 
by setting $\{a_\pe, a_\pf, a_\pg, a_\ph \} = \{-\frac{1}{2}(\omega_r-\delta)\  (r=0, 1, 2, 3)\}$.
Since $A^0(x;\bs{a}| \bs{c})=0$ under this specialization,
the terms with $I_0 \neq \emptyset$ in the left-hand side of \eqref{eq:BC set1} vanish.
Similarly, the terms with $K_0 \neq \emptyset$ in the right-hand side become zero.
Replacing the parameter $\delta$ by $\delta/2$, 
we obtain the following theorem.
\begin{theo}\label{Theo:C set1}
For the parameters $\bs{a} = (a_\pa, a_\pb, a_\pc, a_\pd)$, we define the parameters 
$\bs{b} =(b_\pa, b_\pb, b_\pc, b_\pd)$ by $b_p = \delta/2 - a_p\ (p=\pa, \pb, \pc, \pd)$.
Suppose that the balancing condition $\sum_{p=\pa}^\pd a_p = \sum_{p=\pa}^\pd b_p = \delta$
is satisfied. Then the following identity holds
for two sets of variables $z=(z_1, \ldots, z_N)$ and 
$w=(w_1, \ldots, w_N)$ 
for $r=0, 1, \ldots, N$\,$:$ 
\begin{align}
&\sum_{\substack{I \subset \{1, \ldots, N\}\\ |I|=r }}
\sum_{I_+ \cup I_- = I}
\prod_{i \in I_+} B^+(z_i;\bs{a}) \prod_{i \in I_-} B^-(z_i;\bs{a})\nonumber\\
&\cdot
\prod_{\{i, j \} \subset I_+} \frac{[z_i +z_j +\delta]}{[z_i +z_j]}
\prod_{\{i, j \} \subset I_-} \frac{[z_i +z_j -\delta]}{[z_i +z_j]}
\prod_{\substack{i \in I_+ \\ j \in I_-} }\frac{[z_i - z_j +\delta]}{[z_i - z_j]}\nonumber \\
&
\prod_{\substack{i \in I_+ \\ 1 \le k \le N}}
\frac{[z_i \pm w_k]}{[z_i + \frac{\delta}{2} \pm w_k]}
\prod_{\substack{i \in I_- \\ 1 \le k \le N}}
\frac{[z_i \pm w_k]}{[z_i - \frac{\delta}{2} \pm w_k]}\nonumber \\
&=(-1)^r
\sum_{\substack{K \subset \{1, \ldots, N\}\\ |K|=r }}
\sum_{K_+ \cup K_-  = K}
\prod_{k \in K_+} B^+(w_k; \bs{b}) 
\prod_{k \in K_-} B^-(w_k; \bs{b})\nonumber\\
&\cdot
\prod_{\{k, l \} \subset K_+} \frac{[w_k +w_l +\delta]}{[w_k+w_l]} 
\prod_{\{k, l \} \subset K_-} \frac{[w_k+w_l- \delta]}{[w_k+w_l]}
\prod_{\substack{k \in K_+ \\ l \in K_-}}\frac{[w_k-w_l +\delta]}{[w_k-w_l]} \nonumber\\
&\cdot
\prod_{\substack{k \in K_+ \\ 1 \le i \le N}}
\frac{[w_k \pm z_i]}{[w_k +\frac{\delta}{2} \pm z_i]}
\prod_{\substack{k \in K_- \\ 1 \le i \le N}}
\frac{[w_k \pm z_i]}{[w_k -\frac{\delta}{2} \pm z_i]}. \label{eq:C set1}
\end{align}
\end{theo} 
\section{Duality transformation formulas of type $C$}
\subsection{Duality transformations of type $C$ on multi-indices}
To derive a duality transformation formula for type $C$,
we apply in advance the shift operator $\prod_{i=1}^N T_{z_i}^{\delta/2} \prod_{k=1}^N T_{w_k}^{\delta/2}$
to \eqref{eq:C set1} with $r=N$:
\begin{align}
&\sum_{I_+ \cup I_- = \{1, \ldots, N\}}
\prod_{i \in I_+} B^+(z_i+\tfrac{\delta}{2};\bs{a}) \prod_{i \in I_-} B^-(z_i+\tfrac{\delta}{2};\bs{a}) \nonumber\\
&\cdot
\prod_{\{i, j \} \subset I_+} \frac{[z_i +z_j +2\delta]}{[z_i +z_j+\delta]}
\prod_{\{i, j \} \subset I_-} \frac{[z_i +z_j]}{[z_i +z_j+\delta]}
\prod_{\substack{i \in I_+ \\ j \in I_-} }\frac{[z_i - z_j +\delta]}{[z_i - z_j]}\nonumber \\
&\cdot 
\prod_{\substack{i \in I_+ \\ 1 \le k \le N}}
\frac{[z_i +\frac{\delta}{2} \pm (w_k+\frac{\delta}{2})]}{[z_i +\delta \pm (w_k+\frac{\delta}{2})]}
\prod_{\substack{i \in I_- \\ 1 \le k \le N}}
\frac{[z_i +\frac{\delta}{2} \pm (w_k+\frac{\delta}{2})]}{[z_i  \pm (w_k+\frac{\delta}{2})]} \nonumber \\
&=(-1)^N
\sum_{K_+ \cup K_-  = \{1, \ldots, N\}}
\prod_{k \in K_+} B^+(w_k+\tfrac{\delta}{2}; \bs{b}) 
\prod_{k \in K_-} B^-(w_k+\tfrac{\delta}{2}; \bs{b}) \nonumber\\
&\cdot
\prod_{\{k, l \} \subset K_+} 
\frac{[w_k +w_l +2\delta]}
{[w_k+w_l+\delta]} 
\prod_{\{k, l \} \subset K_-} 
\frac{[w_k+w_l]}{[w_k+w_l+\delta]}
\prod_{\substack{k \in K_+ \\ l \in K_-}}
\frac{[w_k-w_l +\delta]}{[w_k-w_l]} \nonumber\\
&\cdot
\prod_{\substack{k \in K_+ \\ 1 \le i \le N}}
\frac{[w_k+\frac{\delta}{2} \pm (z_i+\frac{\delta}{2})]}
{[w_k +\delta \pm (z_i+\frac{\delta}{2})]}
\prod_{\substack{k \in K_- \\ 1 \le i \le N}}
\frac{[w_k +\frac{\delta}{2} \pm (z_i+\frac{\delta}{2})]}
{[w_k \pm (z_i+\frac{\delta}{2})]}. \label{eq:C set2}
\end{align}
We take two multi-indices
$\alpha =(\alpha_1, \ldots, \alpha_m) \in \mathbb{N}^m$ and
$\beta=(\beta_1, \ldots, \beta_n) \in \mathbb{N}^n$ with
$|\alpha| = |\beta|= N$,
where $|\alpha|=\sum_{i=1}^m \alpha_i$ and 
$|\beta|=\sum_{k=1}^n \beta_k$.
We specialize the variables in \eqref{eq:C set2} by setting $z=(x)_\alpha, x=(x_1, \ldots, x_m)$ and $w=(y)_\beta, y=(y_1, \ldots, y_n)$
as follows:
\begin{equation}\label{eq:mult principal sp}
\begin{split}
z&=(x)_\alpha :=
(x_1, x_1+\delta, \ldots, x_1+(\alpha_1-1)\delta;
\ldots \ldots; x_m, x_m+\delta, \ldots, x_m+(\alpha_m-1)\delta), \\
w&=(y)_\beta :=
(y_1, y_1+\delta, \ldots, y_1+(\beta_1-1)\delta;
\ldots \ldots; y_n, y_n+\delta, \ldots, y_n+(\beta_n-1)\delta). 
\end{split}
\end{equation}
This specialization is called the \textit{multiple principal specialization}.

We first consider the principal specialization $z = (x, x+ \delta, \ldots, x+(N-1)\delta)$ of a single block.
We replace the index set $\{1, \ldots, N\}$ by $I:=\{0, 1, \ldots, N-1 \}$. 
Then \eqref{eq:C set2} is expressed as
\begin{align}
&\sum_{\epsilon \in \{\pm \}^I}
\prod_{i \in I} B^{\epsilon_i} (z_i+\tfrac{\delta}{2};\bs{a}) 
\prod_{\{i, j\} \subset I} 
\frac{[(z_i+\tfrac{\delta}{2}+\epsilon_i \tfrac{\delta}{2}) \pm (z_j +\tfrac{\delta}{2}+\epsilon_j \tfrac{\delta}{2})]}
{[(z_i +\tfrac{\delta}{2}) \pm (z_j +\tfrac{\delta}{2})]} \nonumber \\
&\cdot
\prod_{\substack{i \in I \\ k \in I}} \frac{[(z_i +\tfrac{\delta}{2}) \pm (w_k+\tfrac{\delta}{2})]}
{[(z_i +\tfrac{\delta}{2}+\epsilon_i \tfrac{\delta}{2}) \pm (w_k +\tfrac{\delta}{2})]} \nonumber \\
&=
(-1)^N
\sum_{\epsilon \in  \{ \pm \}^I}
\prod_{k \in I} B^{\epsilon_k} (w_k +\tfrac{\delta}{2}; \bs{b}) \prod_{\{k, l\} \subset I} 
\frac{[(w_k+ \tfrac{\delta}{2} +\epsilon_k \tfrac{\delta}{2}) \pm (w_l +\tfrac{\delta}{2} +\epsilon_l \tfrac{\delta}{2})]}
{[(w_k +\tfrac{\delta}{2}) \pm (w_l +\tfrac{\delta}{2})]} \nonumber \\
&\cdot
\prod_{\substack{k \in I \\ i \in I}} 
\frac{[(w_k +\tfrac{\delta}{2}) \pm (z_i+\tfrac{\delta}{2})]}
{[(w_k +\tfrac{\delta}{2}+\epsilon_k \tfrac{\delta}{2}) \pm (z_i +\tfrac{\delta}{2})]}.
\label{eq:C set epsilon}
\end{align}
In the left-hand side of \eqref{eq:C set epsilon},
the term corresponding to $\epsilon=(\epsilon_0, \ldots, \epsilon_{N-1})$ vanishes if the sign sequence includes
the pattern $(\epsilon_{i}, \epsilon_{i+1}) =(+, -)\ (i=0, 1, \ldots, N-2)$.
If it does not contain the pattern $+-$,
it is an increasing sequence
\begin{equation*}
(\epsilon_0, \ldots, \epsilon_{N-1})=(- \cdots -  \overset{\mu}{+} \cdots +).
\end{equation*}
Then such a sequence $\epsilon=(\epsilon_0, \ldots, \epsilon_{N-1})$
is determined by a non-negative integer $\mu$ such that
$0 \le \mu \le N$;
the number of $-$ signs is given by $\mu$.
Namely,
\begin{equation*}
I_-=[0, \mu), \quad I_+=[\mu, N).
\end{equation*}

When we apply the multiple principal specialization $z=(x)_{\alpha}, w=(y)_{\beta}$ as in \eqref{eq:mult principal sp},
we replace the index set $\{1, \ldots,  N \}$ of the $z$ variables by
\begin{equation*}
\{1, \ldots , N\} = \{(i, a) | 1 \le i \le m, \ 0 \le a < \alpha_i \} 
\end{equation*}
and write $z_{(i, a)} = x_i +a\delta $.
Then all the pairs $(I_-, I_+)$ of 
subsets giving rise to non-zero terms  are
parametrized by the multi-indices $0 \le \mu \le \alpha$,
namely $0 \le \mu_i \le \alpha_i\ (i=1, \ldots, m)$,
as follows:
\begin{equation*}
\begin{split}
I_- &= \{(i, a) | 1 \le i \le m, \ 0 \le a < \mu_i \} ,\\
I_+ &= \{(i, a) | 1 \le i \le m, \ \mu_i \le a < \alpha_i \}. 
\end{split}
\end{equation*}
Specializing the left-hand side of formula \eqref{eq:C set2}
by $z=(x)_\alpha$ and $w=(y)_\beta$ with this parametrization, we have
\begin{align}
&\prod_{i=1}^m \frac{\prod_{p=0}^3[x_i+\frac{\delta}{2}+a_p]_{\alpha_i}}{[2x_i+\delta]_{\alpha_i}}
\prod_{1 \le i < j \le m}
\frac{[x_i+x_j+(\alpha_j+1)\delta]_{\alpha_i} }
{[x_i+x_j+\delta]_{\alpha_i}}
\prod_{\substack{1 \le i \le m \\ 1 \le k \le n}}
\prod_{r=0}^{\beta_k-1}
\frac{[x_i+\frac{\delta}{2}\pm (y_k +  r\delta+ \frac{\delta}{2})]_{\alpha_i}}
{[x_i+\delta \pm (y_k + r\delta + \frac{\delta}{2})]_{\alpha_i}}
\nonumber\\
&\cdot
\sum_{0 \le \mu \le \alpha} 
\prod_{i=1}^m \prod_{p=\pa}^\pd
\frac{[x_i+\frac{\delta}{2}-a_p]_{\mu_i}}
{[x_i+\frac{\delta}{2}+a_p]_{\mu_i}}
\prod_{i=1}^m \frac{[2x_i +2\mu_i\delta]}{[2x_i]}
\prod_{1 \le i < j \le m} 
\frac{[(x_i +\mu_i\delta) \pm (x_j+\mu_j\delta)]}{[x_i \pm x_j]} \nonumber\\
&\cdot
\prod_{1 \le i, j \le m}
\frac{[x_i+x_j]_{\mu_i} [x_i-x_j-\alpha_j\delta]_{\mu_i}}
{[x_i+x_j+(\alpha_j+1)\delta]_{\mu_i} [x_i-x_j+\delta]_{\mu_i}}
\prod_{\substack{1 \le i \le m \\ 1 \le k \le n}}
\frac{[x_i+y_k+\frac{\delta}{2}+\beta_k\delta]_{\mu_i} [x_i-y_k+\frac{\delta}{2}]_{\mu_i}}
{[x_i+y_k+\frac{\delta}{2}]_{\mu_i} 
[x_i-y_k+\frac{\delta}{2}- \beta_k\delta]_{\mu_i}} .
\label{eq:left of type C after specialization}
\end{align}
Here $[u]_k=[u][u+\delta] \cdots [u+(k-1)\delta]$ and $[u \pm v]_k=[u+v]_k [u-v]_k$
for $k=0, 1, 2 \ldots$.
\par\medskip
For each $\alpha \in \mathbb{N}^m$,
we introduce the function 
\begin{equation*}
\Phi_{\alpha}(x|u):=
\Phi_{\alpha}(x_1, \ldots, x_m|u_1, \ldots, u_n)
\end{equation*}
in the variables $x$ and $u$ by
\begin{align*}
\Phi_{\alpha}(x|u)
=&
\sum_{0 \le \mu \le \alpha} 
\prod_{i=1}^m \frac{[2x_i +2\mu_i\delta]}{[2x_i]}
\prod_{1 \le i < j \le m} 
\frac{[(x_i +\mu_i\delta) \pm (x_j+\mu_j\delta)]}{[x_i \pm x_j]} \nonumber \\
&\cdot \prod_{1 \le i, j \le m}
\frac{[x_i+x_j]_{\mu_i} [x_i-x_j-\alpha_j\delta]_{\mu_i}}
{[x_i+x_j+(\alpha_j+1)\delta]_{\mu_i} [x_i-x_j+\delta]_{\mu_i}}
\prod_{\substack{1 \le i \le m \\ 1 \le k \le n}}
\frac{[x_i+u_k]_{\mu_i}}
{[x_i-u_k+\delta]_{\mu_i}} 
\end{align*}
We remark that when $m=1 (\alpha=N)$ we have 
\begin{equation}\label{eq:Phi to V}
\Phi_{N}(x|u_1, \ldots, u_n)
= {}_{n+6}V_{n+5} (2x; x+u_1, \ldots, x+u_n,-N\delta),
\end{equation}
where 
\begin{equation*}
{}_{r+5}V_{r+4} \left( a_{0};  a_{1}, \ldots, a_{r} \right)
=\sum_{k=0}^\infty
\frac{[a_{0}+2k\delta]}{[a_{0}]} \frac{[a_{0}]_{k}}{[\delta]_{k}}
\prod_{i=1}^{r}
\frac{[a_i]_{k}}
{[\delta+a_{0}-a_{i}]_{k}}.
\end{equation*}
We also note that
\begin{equation*}
\Phi_{(\alpha_1, \ldots, \alpha_{m-1}, 0)}(x_1, \ldots, x_m|u_1, \ldots, u_n)
=\Phi_{(\alpha_1, \ldots, \alpha_{m-1})}(x_1, \ldots, x_{m-1}|u_1, \ldots, u_n)
\end{equation*}
and
\begin{equation*}
\Phi_{(\alpha_1, \ldots, \alpha_{m})}(x_1, \ldots, x_m|u_1, \ldots, u_{n-1}, \tfrac{\delta}{2})
=\Phi_{(\alpha_1, \ldots, \alpha_{m})}(x_1, \ldots, x_{m}|u_1, \ldots, u_{n-1}).
\end{equation*}
In this notation of $\Phi_{\alpha}(x|u)$,
\eqref{eq:left of type C after specialization}
is expressed as 
\begin{align*}
&\prod_{i=1}^m \frac{\prod_{p=0}^3[x_i+\frac{\delta}{2}+a_p]_{\alpha_i}}{[2x_i+\delta]_{\alpha_i}}
\prod_{1 \le i < j \le m}
\frac{[x_i+x_j+(\alpha_j+1)\delta]_{\alpha_i} }
{[x_i+x_j+\delta]_{\alpha_i}}
\prod_{\substack{1 \le i \le m \\ 1 \le k \le n}}
\prod_{r=0}^{\beta_k-1}
\frac{[x_i+\frac{\delta}{2}\pm (y_k +  r\delta+ \frac{\delta}{2})]_{\alpha_i}}
{[x_i+\delta \pm (y_k + r\delta + \frac{\delta}{2})]_{\alpha_i}} \nonumber\\
&\cdot
\Phi_{\alpha}
\left(x \middle|
(\tfrac{\delta}{2}-a_p)_{\pa \le p \le \pd},
(\tfrac{\delta}{2}+y_k+\beta_k\delta, \tfrac{\delta}{2}-y_k)_{1 \le k \le n}
\right).
\end{align*}
Similarly, by applying the same specialization to the right-hand side of 
\eqref{eq:C set2} we obtain the following Theorem.
\begin{theo}\label{Theo:C dual M=N}
Let $\alpha \in \mathbb{N}^m, \beta \in \mathbb{N}^n$ be two multi-indices with $|\alpha|=|\beta|=M$.
We assume the balancing condition $\sum_{p=\pa}^\pd a_p =\delta$.
We define the parameters $b_p\ (p=\pa, \pb, \pc, \pd)$ by 
$b_p=\frac{\delta}{2}-a_p\ (p=\pa, \pb, \pc, \pd)$.
For two sets of variables $x=(x_1, \ldots, x_m)$ and
$y=(y_1, \ldots, y_n)$,
the following identity holds\,$:$
\begin{align*}
&\prod_{i=1}^m \frac{\prod_{p=0}^3[x_i+\frac{\delta}{2}+a_p]_{\alpha_i}}{[2x_i+\delta]_{\alpha_i}}
\prod_{1 \le i < j \le m}
\frac{[x_i+x_j+(\alpha_j+1)\delta]_{\alpha_i} }
{[x_i+x_j+\delta]_{\alpha_i}}\nonumber\\
&\cdot
\Phi_{\alpha}
\left(x \middle|
(\tfrac{\delta}{2}-a_p)_{\pa \le p \le \pd},
(\tfrac{\delta}{2}+y_k+\beta_k\delta, \tfrac{\delta}{2}-y_k)_{1 \le k \le n}
\right) \nonumber\\
&=
(-1)^M
\prod_{\substack{1 \le k \le n \\ 1 \le i \le m}}
\frac{[y_k-x_i+\frac{\delta}{2}-\alpha_i \delta]_{\beta_k}}
{[y_k-x_i+\frac{\delta}{2}]_{\beta_k}} \nonumber\\
&\cdot
\prod_{k=1}^n \frac{\prod_{p=0}^3[y_k+\frac{\delta}{2}+b_p]_{\beta_k}}{[2y_k+\delta]_{\beta_k}}
\prod_{1 \le k < l \le n}
\frac{[y_k+y_l+(\beta_l+1)\delta]_{\beta_k} }
{[y_k+y_l+\delta]_{\beta_k}}
\nonumber \\
&\cdot
\Phi_{\beta}\left(y \middle|\,
(\tfrac{\delta}{2}-b_p)_{\pa \le p \le \pd},
( \tfrac{\delta}{2}+x_i +\alpha_i \delta, \tfrac{\delta}{2}-x_i)_{1 \le i \le m} \right).
\end{align*}
\end{theo}
We now  generalize this formula to the case where $|\alpha| \neq |\beta|$. 
In the notation of Theorem \ref{Theo:C dual M=N},
we choose $y_n = b_\pa-\frac{\delta}{2}=-a_\pa$
and set $\sum_{k=1}^{n-1}\beta_k=N$ and $\beta_n=M-N$.
In the left-hand side, we have
\begin{align*}
&\Phi_{\alpha}\left(x \middle|\,
(\tfrac{\delta}{2}-a_p)_{\pa \le p \le \pd},
(\tfrac{\delta}{2}+y_k + \beta_k\delta, \tfrac{\delta}{2}-y_k )_{1 \le k \le n-1},
\tfrac{\delta}{2}-a_\pa+(M-N)\delta, \tfrac{\delta}{2}+a_\pa \right) \nonumber \\
&=\Phi_{\alpha}\left(x \middle|\,
\tfrac{\delta}{2}-a_\pa+(M-N)\delta,
(\tfrac{\delta}{2}-a_p)_{\pb \le p \le \pd},
(\tfrac{\delta}{2}+y_k + \beta_k\delta, \tfrac{\delta}{2}-y_k )_{1 \le k \le n-1} \right).
\end{align*}
On the other hand, in the right-hand side we have 
\begin{align*}
&\Phi_{\beta}\left(y_1, \ldots, y_{n-1}, b_0-\tfrac{\delta}{2} \middle|\, 
(\tfrac{\delta}{2}-b_p)_{\pa \le p \le \pd},
(\tfrac{\delta}{2}+x_i + \alpha_i \delta, \tfrac{\delta}{2}-x_i)_{1 \le i \le m} \right) \nonumber\\
&=\Phi_{(\beta_1, \ldots, \beta_{n-1})}\left(y_1, \ldots, y_{n-1} \middle|\,
\tfrac{\delta}{2}-b_\pa-(M-N)\delta,
(\tfrac{\delta}{2}-b_p)_{\pb \le p \le \pd},
(\tfrac{\delta}{2}+x_i + \alpha_i \delta, \tfrac{\delta}{2}-x_i )_{1 \le i \le m} \right)
\end{align*}
after non-trivial cancellations.
Finally, we have
\begin{align*}
&\prod_{i=1}^m \frac{[x_i+\frac{\delta}{2}+a_0-(M-N)\delta]_{\alpha_i}\prod_{p=\pb}^\pd[x_i+\frac{\delta}{2}+a_p]_{\alpha_i}}
{[2x_i+\delta]_{\alpha_i}}
\prod_{1 \le i < j \le m}
\frac{[x_i+x_j+(\alpha_j+1)\delta]_{\alpha_i} }
{[x_i+x_j+\delta]_{\alpha_i}} \nonumber\\
&\cdot
\Phi_{\alpha}\left(x \middle|\,
\tfrac{\delta}{2}-a_\pa+(M-N)\delta,
(\tfrac{\delta}{2}-a_p)_{\pb \le p \le \pd},
(\tfrac{\delta}{2}+y_k + \beta_k\delta, \tfrac{\delta}{2}-y_k )_{1 \le k \le n-1} \right) \nonumber\\
&=(-1)^M \prod_{p=\pb}^\pd [\delta-a_\pa-a_p]_{M-N}
\prod_{\substack{1 \le k \le n-1 \\ 1 \le i \le m}}
\frac{[y_k-x_i+\frac{\delta}{2}-\alpha_i \delta]_{\beta_k}}
{[y_k-x_i+\frac{\delta}{2}]_{\beta_k}}\nonumber\\
&\cdot
\prod_{k=1}^n
\frac{[y_k+\frac{\delta}{2}+b_0+(M-N)\delta]_{\beta_k}\prod_{p=\pb}^\pd[y_k+\frac{\delta}{2}+b_p]_{\beta_k}}
{[2y_k+\delta]_{\beta_k}}
\prod_{1 \le k< l \le n-1}
\frac{[y_k+y_l+(\beta_l+1)\delta]_{\beta_k}}
{[y_k+y_l+\delta]_{\beta_k}}
\nonumber\\
&\cdot \Phi_{(\beta_1, \ldots, \beta_{n-1})}\left(y_1, \ldots, y_{n-1} \middle|\,
\tfrac{\delta}{2}-b_\pa-(M-N)\delta,
(\tfrac{\delta}{2}-b_p)_{\pb \le p \le \pd},
(\tfrac{\delta}{2}+x_i + \alpha_i \delta, \tfrac{\delta}{2}-x_i )_{1 \le i \le m} \right).
\end{align*}
Replacing $n-1, a_\pa-(M-N)\delta, b_\pa+(M-N)\delta$ by $ n,  a_\pa, b_\pa$ respectively,
we have a duality transformation formula for type $C$ in the case $|\alpha| \ge |\beta|$:
\begin{align}
&\prod_{i=1}^m \frac{\prod_{p=\pa}^\pd[x_i+\frac{\delta}{2}+a_p]_{\alpha_i}}
{[2x_i+\delta]_{\alpha_i}}
\prod_{1 \le i < j \le m}
\frac{[x_i+x_j+(\alpha_j+1)\delta]_{\alpha_i} }
{[x_i+x_j+\delta]_{\alpha_i}}\nonumber\\
&\cdot
\Phi_{\alpha}\left(x \middle|\,
(\tfrac{\delta}{2}-a_p)_{\pa \le p \le \pd},
(\tfrac{\delta}{2}+y_k + \beta_k\delta, \tfrac{\delta}{2}-y_k )_{1 \le k \le n} \right) \nonumber\\
&=(-1)^N \prod_{p=\pb}^\pd [a_\pa+a_p]_{M-N}
\prod_{\substack{1 \le k \le n \\ 1 \le i \le m}}
\frac{[y_k-x_i+\frac{\delta}{2}-\alpha_i \delta]_{\beta_k}}
{[y_k-x_i+\frac{\delta}{2}]_{\beta_k}}\nonumber\\
&\cdot \prod_{k=1}^n
\frac{\prod_{p=\pa}^\pd[y_k+\frac{\delta}{2}+b_p]_{\beta_k}}
{[2y_k+\delta]_{\beta_k}}
\prod_{1 \le k< l \le n}
\frac{[y_k+y_l+(\beta_l+1)\delta]_{\beta_k}}
{[y_k+y_l+\delta]_{\beta_k}} \nonumber \\
&\cdot \Phi_{\beta}\left(y \middle|\,
(\tfrac{\delta}{2}-b_p)_{\pa \le p \le \pd},
(\tfrac{\delta}{2}+x_i + \alpha_i \delta, \tfrac{\delta}{2}-x_i )_{1 \le i \le m} \right),
\label{eq:dual transformation of type C ver Phi}
\end{align}
under the balancing condition $\sum_{p=\pa}^\pd a_p =(N-M+1)\delta$.
Further, since
\begin{equation*}
\prod_{p=\pb}^\pd [a_\pa+a_p]_{M-N}
= \frac{\prod_{p=\pb}^\pd [a_\pa+a_p]_{M}}
{\prod_{1 \le p < q \le \pd}  [b_p+b_q]_{N}}
\end{equation*}
under the balancing condition $\sum_{p=\pa}^\pd a_p =(N-M+1)\delta$,
we find that
 \eqref{eq:dual transformation of type C ver Phi} is expressed as follows:
 \begin{align}
&\prod_{1 \le p < q \le \pd}  [b_p+b_q]_{N}
\prod_{i=1}^m \frac{\prod_{p=\pa}^\pd[x_i+\frac{\delta}{2}+a_p]_{\alpha_i}}
{[2x_i+\delta]_{\alpha_i}}
\prod_{1 \le i < j \le m}
\frac{[x_i+x_j+(\alpha_j+1)\delta]_{\alpha_i} }
{[x_i+x_j+\delta]_{\alpha_i}}\nonumber\\
&\cdot
\Phi_{\alpha}\left(x \middle|\,
(\tfrac{\delta}{2}-a_p)_{\pa \le p \le \pd},
(\tfrac{\delta}{2}+y_k + \beta_k\delta, \tfrac{\delta}{2}-y_k )_{1 \le k \le n} \right) \nonumber\\
&=(-1)^N
\prod_{p=\pb}^\pd [a_\pa+a_p]_{M}
\prod_{\substack{1 \le k \le n \\ 1 \le i \le m}}
\frac{[y_k-x_i+\frac{\delta}{2}-\alpha_i \delta]_{\beta_k}}
{[y_k-x_i+\frac{\delta}{2}]_{\beta_k}} \nonumber\\
&\cdot  
\prod_{k=1}^n
\frac{\prod_{p=\pa}^\pd[y_k+\frac{\delta}{2}+b_p]_{\beta_k}}
{[2y_k+\delta]_{\beta_k}}
\prod_{1 \le k< l \le n}
\frac{[y_k+y_l+(\beta_l+1)\delta]_{\beta_k}}
{[y_k+y_l+\delta]_{\beta_k}} \nonumber\\
&\cdot
\Phi_{\beta}\left(y \middle|\,
(\tfrac{\delta}{2}-b_p)_{\pa \le p \le \pd},
(\tfrac{\delta}{2}+x_i + \alpha_i \delta, \tfrac{\delta}{2}-x_i )_{1 \le i \le m} \right).
\label{eq:dual transformation of type C ver Phi 2}
\end{align}
It is easily checked that this identity is valid in the case $|\alpha| \le |\beta|$ as well.
\begin{theo}[Duality transformation formula of type $C$]
\label{Theo:C dual}
For any complex parameters 
$\bs{a} = (a_\pa, a_\pb, a_\pc, a_\pd)$,
we define $\bs{b} =(b_\pa, b_\pb, b_\pc, b_\pd)$ by $b_p = \frac{\delta}{2} - a_p\ (p=\pa, \pb, \pc, \pd)$.
Take two multi-indices $\alpha \in \mathbb{N}^m$ and 
$\beta \in \mathbb{N}^n$ such that $|\alpha|=M, |\beta|=N$.
Under the balancing condition $\sum_{p=\pa}^\pd a_p =(N-M+1)\delta$,
for two sets of variables $x=(x_1, \ldots, x_m)$ and $y=(y_1, \ldots, y_n)$
the following identity holds\,$:$
\begin{align*}
&\prod_{1 \le p < q \le \pd}  [b_p+b_q]_{N}
\prod_{i=1}^m \frac{\prod_{p=\pa}^\pd[x_i+\frac{\delta}{2}+a_p]_{\alpha_i}}
{[2x_i+\delta]_{\alpha_i}}
\prod_{1 \le i < j \le m}
\frac{[x_i+x_j+(\alpha_j+1)\delta]_{\alpha_i} }
{[x_i+x_j+\delta]_{\alpha_i}}\nonumber\\
&\sum_{0 \le \mu \le \alpha} 
\prod_{i=1}^m \prod_{p=\pa}^\pd
\frac{[x_i+\frac{\delta}{2}- a_p]_{\mu_i}}
{[x_i+\frac{\delta}{2}+a_p]_{\mu_i}}
\prod_{i=1}^m \frac{[2x_i +2\mu_i\delta]}{[2x_i]}
\prod_{1 \le i < j \le m}
\frac{[(x_i +\mu_i\delta) \pm (x_j+\mu_j\delta)]}{[x_i \pm x_j]}\nonumber\\
&\cdot \prod_{1 \le i, j \le m}
\frac{[x_i+x_j]_{\mu_i} [x_i-x_j-\alpha_j\delta]_{\mu_i}}
{[x_i+x_j+(\alpha_j+1)\delta]_{\mu_i} [x_i-x_j+\delta]_{\mu_i}}
\prod_{\substack{1 \le i \le m \\ 1 \le k \le n}}
\frac{[x_i+y_k+\frac{\delta}{2}+\beta_k\delta]_{\mu_i} [x_i-y_k+\frac{\delta}{2}]_{\mu_i}}
{[x_i+y_k+\frac{\delta}{2}]_{\mu_i} 
[x_i-y_k+\frac{\delta}{2}- \beta_k\delta]_{\mu_i}}\nonumber \\
&=(-1)^N  \prod_{p=\pb}^\pd [a_\pa+a_p]_{M}
\prod_{\substack{1 \le k \le n \\ 1 \le i \le m}}
\frac{[y_k-x_i+\frac{\delta}{2}-\alpha_i \delta]_{\beta_k}}
{[y_k-x_i+\frac{\delta}{2}]_{\beta_k}} \nonumber\\
&\cdot \prod_{k =1}^n
\frac{\prod_{p=\pa}^\pd[y_k+\frac{\delta}{2}+b_p]_{\beta_k}}
{[2y_k+\delta]_{\beta_k}}
\prod_{1 \le k< l \le n}
\frac{[y_k+y_l+(\beta_l+1)\delta]_{\beta_k}}
{[y_k+y_l+\delta]_{\beta_k}} \nonumber\\
&
\sum_{0 \le \nu \le \beta}
\prod_{k=1}^n \prod_{p=\pa}^\pd \frac{[y_k+\frac{\delta}{2}-b_p]_{\nu_k}}{[y_k+\frac{\delta}{2}+b_p]_{\nu_k}}
\prod_{k=1}^n \frac{[2y_k+2\nu_k \delta]}{[2y_k]}
\prod_{1 \le k < l \le n} 
\frac{[(y_k+\nu_k\delta)\pm(y_l+\nu_l\delta)]}{[y_k\pm y_l]}
 \nonumber\\
&\cdot
\prod_{1 \le k, l \le n}
\frac{[y_k+y_l]_{\nu_k} [y_k-y_l-\beta_l \delta]_{\nu_k}}
{[y_k+y_l+(\beta_l+1) \delta]_{\nu_k} [y_k-y_l+\delta]_{\nu_k}}
\prod_{\substack{1 \le k \le n \\ 1 \le i \le m}}
\frac{[y_k+x_i+\frac{\delta}{2}+\alpha_i\delta]_{\nu_k}
[y_k+\frac{\delta}{2}-x_i]_{\nu_k}}
{[y_k+x_i+\frac{\delta}{2}]_{\nu_k}
[y_k+\frac{\delta}{2}-x_i-\alpha_i \delta]_{\nu_k}} 
\end{align*}
\end{theo}
Through the limiting procedure we see that Theorem \ref{Theo:C dual} holds in the rational and trigonometric cases as well, with $[u]$ specified as in \eqref{eq:three cases}.
In the trigonometric case,
it gives rise to a one parameter extension of the transformation formula for multiple basic hypergeometric series, due to Rosengren \cite[Corollary 4.4]{Ro1}.

By setting $\alpha=(1, 1, \ldots, 1 )$,
$\beta=(1, 1, \ldots, 1)$
and $z_i = x_i+\frac{\delta}{2}, w_k= y_k+ \frac{\delta}{2}$
in Theorem \ref{Theo:C dual}, we obtain the following generalization of the case $r=N$
of Theorem \ref{Theo:C set1}.
\begin{coro}\label{Coro:C set3}
Consider two sets of variables $z=(z_1, \ldots, z_M), w=(w_1, \ldots, w_N)$. 
We assume the balancing condition $\sum_{p=\pa}^\pd a_p = (N-M+1)\delta$.
We define the parameters $\bs{b}=(b_\pa, b_\pb, b_\pc, b_\pd)$ by 
$b_p=\frac{\delta}{2}-a_p\ (p=\pa, \pb, \pc, \pd)$.
Then, the following identity holds\,$:$
\begin{align*}
&\prod_{1 \le p < q \le \pd}  [b_p+b_q]_{N}
\sum_{I_+ \sqcup I_- = \{1, \ldots, M\}}
\prod_{i \in I_+} 
B^+(z_i;\bs{a})
\prod_{i \in I_-}
B^-(z_i;\bs{a})\nonumber\\
&\cdot
\prod_{\{i, j \} \subset I_+} \frac{[z_i +z_j +\delta]}{[z_i +z_j]}
\prod_{\{i, j \} \subset I_-} \frac{[z_i +z_j-\delta]}{[z_i +z_j]}
\prod_{\substack{i \in I_+ \\ j \in I_-} }
\frac{[z_i - z_j +\delta]}{[z_i - z_j]} \nonumber\\
&\cdot
\prod_{\substack{i \in I_+\\ 1 \le k \le N}}
\frac{[z_i \pm w_k]}{[z_i +\frac{\delta}{2} \pm w_k]}
\prod_{\substack{i \in I_- \\ 1 \le k \le N}}
\frac{[z_i  \pm w_k]}{[z_i -\frac{\delta}{2} \pm w_k]} \nonumber\\
&=(-1)^N \prod_{p=\pb}^\pd [a_\pa+a_p]_{M}
\sum_{K_+ \sqcup K_- = \{1, \ldots, N\}}
\prod_{k \in K_+} B^+(w_k; \bs{b})
\prod_{k \in K_-} B^-(w_k; \bs{b})\nonumber\\
&\cdot
\prod_{\{k, l \} \subset K_+} \frac{[w_k +w_l +\delta]}{[w_k+w_l]} 
\prod_{\{k, l \} \subset K_-} \frac{[w_k+w_l-\delta]}{[w_k+w_l]}
\prod_{\substack{k \in K_+ \\ l \in K_-}}
\frac{[w_k-w_l +\delta]}{[w_k-w_l]} \nonumber\\
&\cdot \prod_{\substack{k \in K_+ \\ 1 \le i \le M}}
\frac{[w_k \pm z_i]}{[w_k +\frac{\delta}{2} \pm z_i]}
\prod_{\substack{k \in K_- \\ 1 \le i \le M}}
\frac{[w_k  \pm z_i]}{[w_k -\frac{\delta}{2} \pm  z_i]}.
\end{align*}
\end{coro}
This corollary has already been proved in \cite[Theorem 7.9]{Ra1} by using a different method.
We also remark that Theorem \ref{Theo:C dual}
can be derived from Corollary \ref{Coro:C set3}
by multiple principal specialization.
\subsection{Some special cases}
In this subsection, we derive some formulas from duality transformation formula
of type $C$.
Considering Theorem \ref{Theo:C dual}
for $\beta=0$ case, we obtain
a $C_m$ generalization of the Frenkel\,--\,Turaev sum, due to Rosengren \cite[Theorem 7.1]{Ro2}.
\begin{coro}
Take a multi-index $\alpha \in \mathbb{N}^m$ such that $|\alpha|=M$.
Under the balancing condition $\sum_{p=\pa}^\pd a_p =-(M-1)\delta$,
the following identity holds\,$:$
\begin{equation*}
\Phi_{\alpha} \left(x \middle| (\tfrac{\delta}{2}-a_p)_{\pa \le p \le \pd} \right)
=\prod_{p=\pb}^\pd [a_\pa+a_p]_{M}
\prod_{i=1}^m 
\frac{[2x_i+\delta]_{\alpha_i}}
{\prod_{p=\pa}^\pd[x_i+\frac{\delta}{2}+a_p]_{\alpha_i}}
\prod_{1 \le i < j \le m}
\frac{[x_i+x_j+\delta]_{\alpha_i}}
{[x_i+x_j+(\alpha_j+1)\delta]_{\alpha_i}}.
\end{equation*}
\end{coro}
By setting $n=1$ in Theorem \ref{Theo:C dual}, we obtain the following transformation formula.
\begin{coro}\label{Coro:C type n=1}
Let $M$ and $N$ be two non-negative integers.
For any complex parameters 
$\bs{a} = (a_\pa, a_\pb, a_\pc, a_\pd)$,
we define the parameters $\bs{b} =(b_\pa, b_\pb, b_\pc, b_\pd)$ by $b_p = \frac{\delta}{2} - a_p\ (p=\pa, \pb, \pc, \pd)$.
Take a multi-index $\alpha \in \mathbb{N}^m$ with $|\alpha|=M$.
Under the balancing condition $\sum_{p=\pa}^\pd a_p =(N-M+1)\delta$,
for two sets of variables $x=(x_1, \ldots, x_m)$ and $y$
the following identity holds\,$:$
\begin{align*}
&\prod_{1 \le p < q \le \pd}  [b_p+b_q]_{N}
\prod_{i=1}^m \frac{\prod_{p=\pa}^\pd[x_i+\frac{\delta}{2}+a_p]_{\alpha_i}}
{[2x_i+\delta]_{\alpha_i}}
\prod_{1 \le i < j \le m}
\frac{[x_i+x_j+(\alpha_j+1)\delta]_{\alpha_i} }
{[x_i+x_j+\delta]_{\alpha_i}} \nonumber\\
&\cdot
\Phi_{\alpha}
\left(x \middle| (\tfrac{\delta}{2}-a_p)_{\pa \le p \le \pd},
(\tfrac{\delta}{2}+y + N\delta, \tfrac{\delta}{2}-y)\right)\nonumber \\
&=(-1)^N \prod_{p=\pb}^\pd [a_\pa+a_p]_{M}
\prod_{i=1}^m
\frac{[y-x_i+\frac{\delta}{2}-\alpha_i \delta]_{N}}
{[y-x_i+\frac{\delta}{2}]_{N}}
\frac{\prod_{p=\pa}^\pd[y+\frac{\delta}{2}+b_p]_{N}}
{[2y+\delta]_{N}}\nonumber\\
&\cdot
{}_{2m+10}V_{2m+9}
(2y;y+\tfrac{\delta}{2}-b_\pa, \ldots, y+\tfrac{\delta}{2}-b_\pd, 
(\tfrac{\delta}{2}+y-x_i, \tfrac{\delta}{2}+y+x_i+\alpha_i\delta)_{1 \le i \le m}, -N\delta). 
\end{align*}
\end{coro}
When $m=1$ in Corollary \ref{Coro:C type n=1}, we obtain ${}_{12}V_{11}$ transformation formula.
\begin{coro}\label{Coro:12V11 MN}
Let $M$ and $N$ be  two non-negative integers.
Under the balancing condition $\sum_{p=\pa}^\pd a_p = (N-M+1)\delta$,
the following identity holds\,$:$
\begin{align*}
&{}_{12}V_{11}(2x;x+\tfrac{\delta}{2}-a_\pa,
x+\tfrac{\delta}{2}-a_\pb ,x+\tfrac{\delta}{2}-a_\pc,x+\tfrac{\delta}{2}-a_\pd,
x+y+\tfrac{\delta}{2}+N\delta, x-y+\tfrac{\delta}{2}, -M\delta) \nonumber\\
&=\prod_{p=\pb}^\pd \frac{[a_\pa+a_p]_{M}}{[(1-M)\delta-a_\pa-a_p]_N}
\prod_{p=\pa}^\pd 
\frac{[y+\frac{\delta}{2}+b_p]_{N}}
{[x+\frac{\delta}{2}+a_p]_{M}}
\frac{[2x +\delta]_{M}}
{[2y+\delta]_{N}}
\frac{[y-x+\frac{\delta}{2}-M \delta]_{N}}
{[y-x+\frac{\delta}{2}]_{N}} \nonumber\\
&\cdot {}_{12}V_{11}
(2y;y+\tfrac{\delta}{2}-b_\pa,
y+\tfrac{\delta}{2}-b_\pb, y+\tfrac{\delta}{2}-b_\pc,y+\tfrac{\delta}{2}-b_\pd,
y+x+\tfrac{\delta}{2}+M\delta, y-x+\tfrac{\delta}{2}, -N\delta),
\end{align*}
where $b_p =\frac{\delta}{2}-a_p\ (p=\pa, \pb, \pc, \pd)$.
\end{coro}
As far as we have checked, 
these transformation formulas of Corollaries \ref{Coro:C type n=1} and \ref{Coro:12V11 MN}
seem to be new even in the case of basic hypergeometric series. 
We also remark that Corollary \ref{Coro:12V11 MN} can also be proved directly
by iterating the elliptic Bailey transformation.

If $M > N$ and  $a_\pa + a_\pb=-L\delta\ (L=0, 1, 2, \ldots, M-N-1)$ in \eqref{eq:dual transformation of type C ver Phi},
the right-hand side vanishes and hence we obtain
\begin{equation*}
\Phi_{\alpha}\left(x \middle|\,
\tfrac{\delta}{2}+a_\pb+L\delta,
\tfrac{\delta}{2}-a_\pb,  \tfrac{\delta}{2}+a_\pd+(M-N-L-1)\delta, \tfrac{\delta}{2}-a_\pd,  
(\tfrac{\delta}{2}+y_k + \beta_k\delta, \tfrac{\delta}{2}-y_k )_{1 \le k \le n} \right)=0.
\end{equation*}
Regarding $a_\pb, a_\pd$ as additional $y$ variables, we have the following zero formula.
\begin{coro}\label{Coro:zero formula}
Take two multi-indices $\alpha \in \mathbb{N}^m$ and 
$\beta \in \mathbb{N}^n$ such that $|\alpha|= |\beta|+1$.
Then,
\begin{equation*}
\Phi_{\alpha}\left(x \middle|\,
(\tfrac{\delta}{2}+y_k + \beta_k\delta, \tfrac{\delta}{2}-y_k )_{1 \le k \le n} \right)=0.
\end{equation*}
\end{coro}
\subsection{Multiple Karlsson\,--\,Minton type transformation}
In this subsection,
we derive a multiple Karlsson\,--\,Minton type transformation from Theorem \ref{Theo:C dual}.
\begin{theo}\label{Theo:transformation N+r to N+s}
For any non-negative integers $N \ge 0, r \ge 0$ and $s \ge 0$,
take two multi-indices $\alpha \in \mathbb{N}^{m}$ and 
$\beta \in \mathbb{N}^n$ such that $|\alpha|=N+r+s, |\beta|=N$.
For two sets of variables $x=(x_1, \ldots, x_m)$ and $y=(y_1, \ldots, y_n)$
and two parameters $u$ and $v$,
the following identity holds\,$:$
\begin{align}
&
\dfrac{
[v+\delta\pm v]_s
[u+\delta\pm v]_r
\prod_{k=1}^{n}
[y_k+\delta\pm v]_{\beta_k}
}{
\prod_{i=1}^{m}
[x_i+\tfrac{\delta}{2}\pm v]_{\alpha_i}
}
\nonumber\\
&\cdot
\Phi_{(\beta,r)}\left(y,u \middle|\,
v, -v-s\delta,
(\tfrac{\delta}{2}+x_i+\alpha_i\delta, \tfrac{\delta}{2}-x_i)_{1 \le i\le m} \right)
\nonumber\\
&=
\dfrac{
[u+\delta\pm u]_r
[v+\delta\pm u]_s
\prod_{k=1}^{n}
[y_k+\delta\pm u]_{\beta_k}
}{
\prod_{i=1}^{m}
[x_i+\tfrac{\delta}{2}\pm u]_{\alpha_i}
}
\nonumber\\
&\cdot
\Phi_{(\beta,s)} \left(y,v \middle|\,
u, -u-r\delta,
(\tfrac{\delta}{2}+x_i+\alpha_i\delta, \tfrac{\delta}{2}-x_i)_{1\le i\le m} \right).
\label{eq:transformation N+r to N+s}
\end{align}
\end{theo}
\begin{proof}
For any non-negative integers $r ,s$,
we take two multi-indices $\alpha \in \mathbb{N}^m$
and $\beta \in \mathbb{N}^n$ such that $|\alpha|=N+r+s-l$ and $|\beta|=N+r$,
where $l=0, 1, 2, \ldots, r+s$.
We rewrite 
\begin{equation*}
\Phi_\alpha \left(x \middle|\,
\tfrac{\delta}{2}+a-(l+1)\delta, \tfrac{\delta}{2}-a,
(\tfrac{\delta}{2}+u+r\delta, \tfrac{\delta}{2}-u),
(\tfrac{\delta}{2}+v+s\delta,\tfrac{\delta}{2}-v),
(\tfrac{\delta}{2}+y_k+\beta_k\delta, \tfrac{\delta}{2}-y_k)_{1\le k\le n} \right)
\end{equation*}
in two ways,
one by applying \eqref{eq:dual transformation of type C ver Phi 2} with
$\bs{a}=(a, (l+1)\delta-a, v, -s\delta-v)$ and the other
with $\bs{a}=(a, (l+1)\delta-a, u, -r\delta-u)$.
Then we have
\begin{align*}
&\dfrac{
[v+\delta\pm v]_s
[u+\delta\pm v]_r
\prod_{k=1}^{n}
[y_k+\delta\pm v]_{\beta_k}
}{
[\delta-a\pm v]_l
\prod_{i=1}^{m}
[x_i+\tfrac{\delta}{2}\pm v]_{\alpha_i}
}
\nonumber
\\
&\cdot
\Phi_{(\beta,r)}\left(y,u
\middle|\,
v,
-v-s\delta,
(\tfrac{\delta}{2}+x_i+\alpha_i\delta, \tfrac{\delta}{2}-x_i)_{1 \le i\le m},
((l+1)\delta-a, a) \right) \nonumber\\
&=\dfrac{
[u+\delta\pm u]_r
[v+\delta\pm u]_s
\prod_{k=1}^{n}
[y_k+\delta\pm u]_{\beta_k}
}{
[\delta-a\pm u]_l
\prod_{i=1}^{m}
[x_i+\tfrac{\delta}{2}\pm u]_{\alpha_i}
}
\nonumber\\
&\cdot
\Phi_{(\beta,s)}
\left(y,v
\middle|\,
u,
-u-r\delta,
(\tfrac{\delta}{2}+x_i+\alpha_i\delta, \tfrac{\delta}{2}-x_i)_{1\le i\le m},
((l+1)\delta-a, a) \right)
\end{align*}
after non-trivial cancellation.
Setting $x_{m+1}= \tfrac{\delta}{2}-a, \alpha_{m+1}=l$
and replacing $m+1$ by $m$,
we obtain formula \eqref{eq:transformation N+r to N+s}.
\end{proof}
By setting $\beta=0$ in Theorem \ref{Theo:transformation N+r to N+s}, 
from \eqref{eq:Phi to V} we obtain the following transformation.
\begin{coro}\label{Coro:transformation 2m+8V2m+7}
For any non-negative integers $r \ge 0$ and $s \ge 0$,
take a multi-index $\alpha \in \mathbb{N}^{m}$ such that $|\alpha|=r+s$.
For the variables $x=(x_1, \ldots, x_m)$ and two parameters $u$ and $v$,
the following identity holds\,$:$
\begin{align*}
&
\dfrac{
[v+\delta\pm v]_s
[u+\delta\pm v]_r
}{
\prod_{i=1}^{m}
[x_i+\tfrac{\delta}{2}\pm v]_{\alpha_i}
}
\nonumber\\
&\cdot
{}_{2m+8}V_{2m+7}(2u;u+v, u-v-s\delta,
(u+\tfrac{\delta}{2}+x_i+\alpha_i\delta, u+\tfrac{\delta}{2}-x_i)_{1 \le i\le m}
,-r\delta)
\nonumber\\
&=
\dfrac{
[u+\delta\pm u]_r
[v+\delta\pm u]_s
}{
\prod_{i=1}^{m}
[x_i+\tfrac{\delta}{2}\pm u]_{\alpha_i}
}
\nonumber\\
&\cdot
{}_{2m+8}V_{2m+7}(2v;v+u, v-u-r\delta,
(v+\tfrac{\delta}{2}+x_i+\alpha_i\delta, v+\tfrac{\delta}{2}-x_i)_{1 \le i\le m}
,-s\delta). 
\end{align*}
\end{coro}
By setting $s=0$ in Corollary \ref{Coro:transformation 2m+8V2m+7}, 
we obtain the following summation formula.
\begin{coro}\label{Coro:summation 2m+8V2m+7}
Take a multi-index $\alpha \in \mathbb{N}^{m}$ such that $|\alpha|=M$.
For the variables $x=(x_1, \ldots, x_m)$ and two parameters $u$ and $v$,
the following identity holds\,$:$
\begin{align*}
&
{}_{2m+8}V_{2m+7}(2u;u+v, u-v,
(u+\tfrac{\delta}{2}+x_i+\alpha_i\delta, u+\tfrac{\delta}{2}-x_i)_{1 \le i\le m}
,-M\delta)
\nonumber\\
&=
\dfrac{[u+\delta\pm u]_M}
{[u+\delta\pm v]_M}
\prod_{i=1}^{m}
\dfrac{[x_i+\tfrac{\delta}{2}\pm v]_{\alpha_i}}
{[x_i+\tfrac{\delta}{2}\pm u]_{\alpha_i}}.
\end{align*}
\end{coro}
Corollary \ref{Coro:transformation 2m+8V2m+7} and
\ref{Coro:summation 2m+8V2m+7} have been proved by Rosengren and Schlosser in \cite{RS1}.
\section{Duality transformation formulas of type $BC$}
\subsection{Duality transformation of type $BC$ on multi-indices}\label{subsection:Duality transformation of type $BC$ on multi-indices}
In this subsection,
we derive a $BC$ type duality transformation formula from the case $r=N$ of Theorem \ref{Theo:BC set}:
\begin{align}
&\sum_{I_+ \cup I_- \cup I_0 = \{1, \ldots, N \}}
\prod_{i \in I_+} A^+(z_i; \bs{a}) \prod_{i \in I_0} A^0(z_i;\bs{a} | \bs{c})
\prod_{i \in I_-} A^-(z_i;\bs{a})\nonumber\\
&\cdot
\prod_{\{i, j \} \subset I_+} \frac{[z_i +z_j +2\delta]}{[z_i +z_j]}
\prod_{\{i, j \} \subset I_-} \frac{[z_i +z_j -2\delta]}{[z_i +z_j]}
\prod_{\substack{i \in I_+ \\ j \in I_-} }\frac{[z_i - z_j +2\delta]}{[z_i - z_j]} \nonumber\\
&\cdot
\prod_{\substack{i \in I_+ \\ j \in I_0} }
\frac{[z_i + \delta \pm z_j]}{[z_i \pm z_j]}
\prod_{\substack{i \in I_- \\ j \in I_0} }
\frac{[z_i - \delta \pm z_j]}{[z_i \pm z_j]}
\prod_{\substack{i \in I_+ \\ 1 \le k \le N}}
\frac{[z_i \pm w_k]}{[z_i + \delta \pm w_k]}
\prod_{\substack{i \in I_- \\ 1 \le k \le N}}
\frac{[z_i \pm w_k]}{[z_i - \delta \pm w_k]} \nonumber\\
&=
\sum_{K_+ \cup K_- \cup K_0 = \{1, \ldots, N \}}
\prod_{k \in K_+} A^+(w_k;\bs{b}) \prod_{k \in K_0} A^0(w_k;\bs{b} | \bs{c})
\prod_{k \in K_-} A^-(w_k; \bs{b})\nonumber\\
&\cdot
\prod_{\{k, l \} \subset K_+} \frac{[w_k +w_l +2\delta]}{[w_k+w_l]} 
\prod_{\{k, l \} \subset K_-} \frac{[w_k+w_l- 2\delta]}{[w_k+w_l]}
\prod_{\substack{k \in K_+ \\ l \in K_-}}\frac{[w_k-w_l +2\delta]}{[w_k-w_l]} \nonumber\\
&\cdot
\prod_{\substack{k \in K_+ \\ l \in K_0}}
\frac{[w_k+\delta \pm w_l ]}{[w_k \pm w_l]}
\prod_{\substack{k \in K_- \\ l \in K_0}}
\frac{[w_k-\delta \pm w_l ]}{[w_k \pm w_l]} 
\prod_{\substack{k \in K_+ \\ 1 \le i \le N}}
\frac{[w_k \pm z_i]}{[w_k +\delta \pm z_i]}
\prod_{\substack{k \in K_- \\ 1 \le i \le N}}
\frac{[w_k \pm z_i]}{[w_k -\delta \pm z_i]}.
\label{eq:BC set N}
\end{align}
We replace the index set $\{1, \ldots, N\}$ by $I:=\{0, 1, \ldots, N-1 \}$. 
Then \eqref{eq:BC set N} is expressed as
\begin{align}
&\sum_{\epsilon \in \{\pm, 0\}^I}
\prod_{i \in I} A^{\epsilon_i} (z_i;\bs{a}) \prod_{\{i, j\} \subset I} 
\frac{[(z_i+\epsilon_i \delta) \pm (z_j +\epsilon_j \delta)]}{[z_i \pm z_j]} 
\prod_{\substack{i \in I \\ k \in I}} \frac{[z_i \pm w_k]}{[z_i +\epsilon_i \delta \pm w_k]} \nonumber\\
&=\sum_{\epsilon \in  \{ \pm , 0\}^I}
\prod_{k \in I} A^{\epsilon_k} (w_k; \bs{b}) \prod_{\{k, l\} \subset I} 
\frac{[(w_k+\epsilon_k \delta) \pm (w_l +\epsilon_l \delta)]}{[w_k \pm w_l]} 
\prod_{\substack{k \in I \\ i \in I}} \frac{[w_k \pm z_i]}{[w_k +\epsilon_k \delta \pm z_i]}.
\label{eq:BC set N2}
\end{align}
We take two multi-indices
$\alpha =(\alpha_1, \ldots, \alpha_m) \in \mathbb{N}^m$ and
$\beta=(\beta_1, \ldots, \beta_n) \in \mathbb{N}^n$ with
$|\alpha| = |\beta|= N$,
where $|\alpha|=\sum_{i=1}^m \alpha_i$ and 
$|\beta|=\sum_{k=1}^n \beta_k$.
We specialize $z =(z_0, \ldots, z_{N-1})$ and
$w=(w_0 \ldots, w_{N-1})$ in \eqref{eq:BC set N2} as follows:
\begin{equation}\label{eq:multiple principal specialization}
\begin{split}
z&=(x)_\alpha :=
(x_1, x_1+\delta, \ldots, x_1+(\alpha_1-1)\delta;
\ldots \ldots; x_m, x_m+\delta, \ldots, x_m+(\alpha_m-1)\delta), \\
w&=(y)_\beta :=
(y_1, y_1+\delta, \ldots, y_1+(\beta_1-1)\delta;
\ldots \ldots; y_n, y_n+\delta, \ldots, y_n+(\beta_n-1)\delta). 
\end{split}
\end{equation}

We first consider the principal specialization $z = (x, x+ \delta, \ldots, x+(N-1)\delta)$ of a single block.
In the left-hand side of \eqref{eq:BC set N2},
the term corresponding to $\epsilon=(\epsilon_0, \ldots, \epsilon_{N-1})$ vanishes if the sign sequence includes
some of the following patterns:
\begin{equation}\label{eq:condition of sign}
\begin{aligned}
+\, 0 &:\  \epsilon_i=+,\   \epsilon_{i+1}=0 \quad & ( &0 \le i \le N-2), \\
0\, - &:\  \epsilon_i=0,\  \epsilon_{i+1}=- \quad & ( & 0 \le i \le N-2), \\
+\, *\, - &:\  \epsilon_i=+,\  \epsilon_{i+1} \in \{\pm , 0 \},\  \epsilon_{i+2}=- \quad & ( & 0 \le i \le N-3).
\end{aligned}
\end{equation}
Supposing that a sign sequence $\epsilon=(\epsilon_0, \ldots, \epsilon_{N-1})$
does not contain any pattern in \eqref{eq:condition of sign},
we count the number $r$ of occurrences of the pattern $+-$.
If it does not contain the pattern $+-\ (r=0)$,
it is an increasing sequence
\begin{equation*}
(\epsilon_0, \ldots, \epsilon_{N-1})=(- \cdots - \overset{\nu}{0}  \cdots 0 \overset{\mu}{+} \cdots +).
\end{equation*}
Then such a sequence $\epsilon=(\epsilon_0, \ldots, \epsilon_{N-1})$
is determined by two non-negative integers $\nu$ and $\mu$ such that
$0 \le \nu \le \mu \le N$;
the number of $-$ and $+$ signs are given by $\nu$ and by $N-\mu$, respectively.
Namely,
\begin{equation*}
I_-=[0, \nu), \quad I_0=[\nu, \mu), \quad I_+=[\mu, N).
\end{equation*}
When $r >0$,
we number the places of $-$ in the patterns $+-$ as follows:
\begin{equation*}
| \overset{0}{\cdot}\cdot \cdot \cdots | + \overset{i_1}{-} | \cdots | +\overset{i_2}{-} | \cdots \cdots \cdots | + \overset{i_r}{-} | 
\cdots \cdot \cdot \overset{N-1}{\cdot} |,
\end{equation*}
where $\{i_p\}_{1 \le p \le r}$ is an increasing sequence of positive integers such that
\begin{equation*}
0 < i_1 < i_2 < \cdots < i_r < N; \quad i_{p+1}-i_p \ge 2\ (p=1, \ldots, r-1).
\end{equation*}
Then by the assumption that it does not contain the third pattern of \eqref{eq:condition of sign},
we find that $\epsilon_{i_{p}-2}\neq +$ and $\epsilon_{i_{p}+1} \neq -$.
Hence  such a sign sequence $\epsilon = (\epsilon_0, \ldots, \epsilon_{N-1})$ should be of the form
\begin{equation}\label{eq:sequence}
| \underbrace{\overset{0}{-}\cdots -}_{\mu^-} \overset{\nu}{0} \cdots 0 | + \overset{i_1}{-} |0 \cdots 0 |
 +\overset{i_2}{-} |0 \cdots \cdots \cdots 0
| + \overset{i_r}{-} | 0 \cdots 0 \underbrace{\overset{\mu}{+} \cdots \overset{N-1}{+}}_{\mu^+} |.
\end{equation}
We denote the number of $-$ in $[0, i_1-2]$ by $\mu^{-}$ and
the number of $+$ in $[i_{r}+1, N-1]$ by $\mu^+$. 
If we set $\nu= \mu^-$ and $\mu=N-\mu^+$, 
$\epsilon=(\epsilon_0, \ldots, \epsilon_{N-1})$ is determined by the non-negative integer sequence
\begin{equation*}
0 \le \nu < i_1 < i_2 < \cdots < i_r < \mu \le N; \quad i_{p+1}-i_p \ge 2 \quad (p=1, \ldots, r-1).
\end{equation*}
In terms of this sequence,
the corresponding subsets $I_-, I_0, I_+$ are given by
\begin{equation*}
\begin{split}
&I_-=[0, \nu) \cup \{i_p | p=1, \ldots, r \}, \\
&I_0=[\nu, \mu) \setminus \{i_p , i_p-1| p=1, \ldots, r \},\\
&I_+=[\mu, N) \cup \{i_p-1 | p=1, \ldots, r \}.
\end{split}
\end{equation*}

We now analyze the contribution of the $+-$ patterns in the expression
\begin{equation}\label{eq:remark factor}
\prod_{\{i, j\} \subset I} 
\frac{[(z_i+\epsilon_i \delta) \pm (z_j +\epsilon_j \delta)]}{[z_i \pm z_j]} 
\prod_{\substack{i \in I \\ k \in I}} \frac{[z_i \pm w_k]}{[z_i +\epsilon_i \delta \pm w_k]}.
\end{equation}
Supposing that $(\epsilon_{i}, \epsilon_{i+1}) =(+, -)$, we set $(z_i, z_{i+1})=(a, a+\delta)$.
Then
\begin{equation*}
\frac{\Delta(z_i+\epsilon_i\delta, z_{i+1}+\epsilon_{i+1}\delta)}{\Delta(z_i, z_{i+1})} =
\frac{\Delta(a+\delta, a)}{\Delta(a, a+\delta)}=-1,
\end{equation*}
where $\Delta(z_i, z_j) := [z_i \pm z_j]$.
We also obtain
\begin{align*}
\frac{\Delta(z_i+\epsilon_i \delta, z_k+\epsilon_k\delta)}{\Delta(z_i, z_k)}
\frac{\Delta(z_{i+1}+\epsilon_{i+1} \delta, z_k+\epsilon_k\delta)}{\Delta(z_{i+1}, z_k)}
&=
\frac{\Delta(a+\delta, z_k+\epsilon_k \delta)}{\Delta(a, z_k)}
\frac{\Delta(a, z_k+\epsilon_k \delta)}{\Delta(a+\delta, z_k)} \nonumber\\
&=\frac{\Delta(z_i, z_k+\epsilon_k\delta)}{\Delta(z_i, z_k)}
\frac{\Delta(z_{i+1}, z_k+\epsilon_k\delta)}{\Delta(z_{i+1}, z_k)},
\end{align*}
for $k \neq i, i+1$ and
\begin{equation*}
\begin{split}
\frac{\Delta(z_i+\epsilon_i\delta, w_l)}{\Delta(z_i, w_l)}
\frac{\Delta(z_{i+1}+\epsilon_{i+1}\delta, w_l)}{\Delta(z_{i+1}, w_l)}
&=
\frac{\Delta(a+\delta, w_l)}{\Delta(a, w_l)}
\frac{\Delta(a, w_l)}{\Delta(a+\delta, w_l)}=1
\end{split}
\end{equation*}
for $l=0, \ldots, N-1$.
This means that the pattern $+-$ in the pair $(z_i, z_{i+1})$  
gives the same effects as the pattern $00$ to other variables $z_k\ (k \neq i, i+1)$ and $w_l\ (0 \le l \le N-1)$.
Hence we see that 
the expression \eqref{eq:remark factor} for the sign sequence 
$\epsilon=(\epsilon_0, \ldots, \epsilon_{N-1})$ \eqref{eq:sequence} coincides with
that for the increasing sequence $\epsilon=(- \cdots - 0 \cdots 0 + \cdots +)$
obtained by replacing $00$ for $+-$ (with same $\nu, \mu$)
up to multiplication by $(-1)^r$.

We next consider applying the multiple principal specialization 
\eqref{eq:multiple principal specialization} to \eqref{eq:BC set N2}.
We replace the index set $I=\{0, \ldots, N-1\}$ by
\begin{equation*}
I= \bigcup_{i=1}^m I^{(i)}, \quad I^{(i)} = \{(i, k)| k \in [0, \alpha_i) \}
\end{equation*}
and set $z_{(i, k)} := x_i + (k-1)\delta$.
By the same argument above,
we find that
the term corresponding to a sign sequence $\epsilon$ vanishes
if it contains  either $+0, 0+$ or $+*-$ in some block.
For each $i=1, \ldots, m$,
we denote by $r_i$ the number of patterns $+-$ 
in the $i$-th block $I^{(i)}$.
Similarly to the case of a single block,
we number the positions of $-$ in the patterns $+-$
by $\xi_{i, p} \ (1\le i \le m, 1 \le p \le r_i)$.
For two multi-indices $\nu, \mu \in \mathbb{N}^m$,
we write $\nu \le \mu$ 
if $\nu_i \le \mu_i$ for all $i=1, \ldots, m$.
Then the set of sign sequences $\epsilon$ that give rise to non-zero terms is parametrized graphically as
\begin{equation}\label{eq:sing sequence}
|\cdots \cdots \cdots ;
 |\overset{0}{-}\cdots - \overset{\nu_i}{0} \cdots 0 | 
+ \overset{\xi_{i,1}}{-} |0 \cdots 0 |
+\overset{\xi_{i,2}}{-} |0 \cdots \cdots \cdots 0
| + \overset{\xi_{i,r_i}}{-} | 0 \cdots 0 \overset{\mu_i}{+} \cdots \overset{\alpha_i-1}{+}|;
\cdots \cdots \cdots |  
\end{equation}
by two multi-indices $\nu, \mu \in \mathbb{N}^m$ with
$0 \le \nu \le \mu \le \alpha$
and a sequence $(\xi_{i, p})_{1 \le i \le m, \, 1 \le p \le r_i}$ of positive integers such that
\begin{equation*}
0 \le \nu_i < \xi_{i,1} < \xi_{i,2} < \cdots < \xi_{i,r_i} < \mu_{i} \le \alpha_{i}; \quad \xi_{i,p+1}-\xi_{i, p} \ge 2\ 
(1 \le i \le m,\ 1 \le p \le r_i).\\
\end{equation*}
As we have seen above, each $+-$ pattern  in a block
behaves like $00$ in relation to other variables.
Denoting the term corresponding to
\begin{equation*}
|\cdots \cdots \cdots ;
 |\overset{0}{-}\cdots - \overset{\nu_i}{0} \cdots 
 \cdots 0 \overset{\mu_i}{+} \cdots \overset{\alpha_i-1}{+}|;
\cdots \cdots \cdots |
\end{equation*}
by $F^{\alpha}_{\mu, \nu}(x;y)$,
we find that
the term corresponding to the sign sequence \eqref{eq:sing sequence} is equal to 
\begin{equation*}
(-1)^{\sum_i r_i}
\prod_{i=1}^m
\prod_{p=1}^{r_i}
\frac{A^+(x_i+(\xi_{i,p}-1)\delta;\bs{a})A^-(x_i+\xi_{i,p}\delta;\bs{a})}
{A^0(x_i+(\xi_{i,p}-1)\delta;\bs{a} | \bs{c})A^0(x_i+\xi_{i,p}\delta;\bs{a} | \bs{c})}
F^{\alpha}_{\mu, \nu}(x;y).
\end{equation*}
By using the formula
\begin{equation*}
\prod_{1 \le i \neq j \le m}\frac{[x_i-x_j-\mu_j \delta]_{\mu_i}}{[x_i -x_j +\delta]_{\mu_i}}
= \prod_{1 \le i < j \le m}\frac{[x_i-x_j]}{[(x_i+\mu_i\delta)-(x_j+\mu_j \delta)]},
\end{equation*}
we can explicitly compute $F^{\alpha}_{\mu, \nu}(x;y)$ as follows:
\begin{align*}
F^{\alpha}_{\mu, \nu}(x;y)
&=(-1)^{|\nu|+|\alpha-\mu|}
\prod_{i=1}^m A^-(x_i;\bs{a})_{\nu_i} A^0(x_i+\nu_i \delta;\bs{a} | \bs{c})_{\mu_i-\nu_i}
A^+(x_i+\mu_i\delta;\bs{a})_{\alpha_i -\mu_i} \nonumber \\
&\cdot \prod_{i=1}^m
\frac{[2x_i+2(\nu_i-1)\delta]}{[2x_i-2\delta]}
\frac{[2x_i+2\mu_i\delta]}{[2x_i+2\alpha_i\delta]} \nonumber\\
&\cdot\prod_{1 \le i< j \le m}
\frac{[(x_i+\nu_i\delta-\delta) \pm (x_j+\nu_j\delta -\delta)]}
{[(x_i-\delta) \pm (x_j-\delta)]}
\frac{[(x_i+\mu_i\delta) \pm (x_j+\mu_j\delta)]}
{[(x_i+\alpha_i \delta) \pm (x_j+\alpha_j\delta)]}\nonumber\\
&\cdot \prod_{1 \le i, j \le m}
\frac{[(x_i-\delta) \pm (x_j+\alpha_j\delta)]}
{[(x_i-\delta) \pm (x_j+\mu_j \delta)]}
\frac{[(x_i+\nu_i\delta-\delta) \pm (x_j+\mu_j\delta)]}
{[(x_i+\nu_i\delta-\delta) \pm (x_j+\alpha_j\delta)]} \nonumber\\
&\cdot\prod_{1 \le i, j \le m}
\frac{[x_i+ x_j-2\delta]_{\nu_i}}
{[x_i+x_j+(\alpha_j -1)\delta]_{\nu_i}}
\frac{[x_i- x_j-\alpha_j \delta]_{\nu_i}}
{[x_i-x_j+ \delta]_{\nu_i}}\nonumber\\
&\cdot \prod_{1 \le i, j \le m}
\frac{[-x_i- x_j-(\alpha_i+\alpha_j)\delta]_{\alpha_i-\mu_i}}
{[-x_i- x_j-(\alpha_i-1)\delta]_{\alpha_i-\mu_i}}
\frac{[-x_i+ x_j-\alpha_i\delta]_{\alpha_i-\mu_i}}
{[-x_i + x_j + (\alpha_j-\alpha_i+1)\delta]_{\alpha_i-\mu_i}}\nonumber\\
&\cdot
\prod_{\substack{1 \le i \le m \\ 1 \le k \le n}}
\frac{[x_i+y_k+(\beta_k-1)\delta]_{\nu_i}}
{[x_i+y_k-\delta]_{\nu_i}}
\frac{[x_i-y_k]_{\nu_i}}
{[x_i-y_k-\beta_k\delta]_{\nu_i}} \nonumber\\
&\cdot
\prod_{\substack{1 \le i \le m \\ 1 \le k \le n}}
\frac{[-x_i-y_k-(\alpha_i-1)\delta]_{\alpha_i-\mu_i}}
{[-x_i-y_k-(\alpha_i+\beta_k-1)\delta]_{\alpha_i-\mu_i}}
\frac{[-x_i+y_k-(\alpha_i-\beta_k)\delta]_{\alpha_i-\mu_i}}
{[-x_i+y_k-\alpha_i\delta]_{\alpha_i-\mu_i}},
\end{align*}
where we used the shorthand notation $f(u)_{k}:= \prod_{i=0}^{k-1} f(u+i \delta)\ (k =0, 1, 2, \ldots )$.
Applying the same specialization to the right-hand side of \eqref{eq:BC set N2},
we obtain the following duality transformation formula of type $BC$. 
\begin{theo}[Duality transformation formula of type $BC$]
\label{Theo:BC multi1}
For any complex parameters $\bs{a} = (a_\pa, a_\pb, \ldots, a_\ph)$,
we define the parameters $\bs{b} =(b_\pa, b_\pb, \ldots, b_\ph)$ by $b_p = \delta - a_p\ (p=\pa, \pb, \ldots, \ph)$.
Take two multi-indices $\alpha \in \mathbb{N}^m$ and 
$\beta \in \mathbb{N}^n$ with $|\alpha|=|\beta|$.
Under the balancing condition $\sum_{p=\pa}^\ph a_p =4\delta$,
the following identity holds
for two sets of variables $x=(x_1, \ldots, x_m)$ and $y=(y_1, \ldots, y_n)$\,$:$
\begin{align*}
&\sum_{0 \le \nu \le \mu \le  \alpha}
(-1)^{|\nu|+|\alpha-\mu|}
\prod_{i=1}^m 
A^-(x_i;\bs{a})_{\nu_i} 
A^+(x_i+\mu_i\delta;\bs{a})_{\alpha_i -\mu_i} C_{{\mu_i}-{\nu_i}}(x_i+\nu_i\delta;\bs{a} | \bs{c})  \nonumber\\
&\cdot 
\prod_{i=1}^m
\frac{[2x_i+2(\nu_i-1)\delta]}{[2x_i-2\delta]}
\frac{[2x_i+2\mu_i\delta]}{[2x_i+2\alpha_i\delta]}\nonumber\\
&\cdot \prod_{1 \le i< j \le m}
\frac{[(x_i+\nu_i\delta-\delta) \pm (x_j+\nu_j\delta -\delta)]}
{[(x_i-\delta) \pm (x_j-\delta)]}
\frac{[(x_i+\mu_i\delta) \pm (x_j+\mu_j\delta)]}
{[(x_i+\alpha_i \delta) \pm (x_j+\alpha_j\delta)]}\nonumber\\
&\cdot \prod_{1 \le i, j \le m}
\frac{[(x_i-\delta) \pm (x_j+\alpha_j\delta)]}
{[(x_i-\delta) \pm (x_j+\mu_j \delta)]}
\frac{[(x_i+\nu_i\delta-\delta) \pm (x_j+\mu_j\delta)]}
{[(x_i+\nu_i\delta-\delta) \pm (x_j+\alpha_j\delta)]}
\frac{[x_i+ x_j-2\delta]_{\nu_i}}
{[x_i+x_j+(\alpha_j -1)\delta]_{\nu_i}}
\frac{[x_i- x_j-\alpha_j \delta]_{\nu_i}}
{[x_i-x_j+ \delta]_{\nu_i}}\nonumber\\
&\cdot \prod_{1 \le i, j \le m}
\frac{[-x_i- x_j-(\alpha_i+\alpha_j)\delta]_{\alpha_i-\mu_i}}
{[-x_i- x_j-(\alpha_i-1)\delta]_{\alpha_i-\mu_i}}
\frac{[-x_i+ x_j-\alpha_i\delta]_{\alpha_i-\mu_i}}
{[-x_i + x_j + (\alpha_j-\alpha_i+1)\delta]_{\alpha_i-\mu_i}}\nonumber\\
&\cdot
\prod_{\substack{1 \le i \le m \\ 1 \le k \le n}}
\frac{[x_i+y_k+(\beta_k-1)\delta]_{\nu_i}}
{[x_i+y_k-\delta]_{\nu_i}}
\frac{[x_i-y_k]_{\nu_i}}
{[x_i-y_k-\beta_k\delta]_{\nu_i}} \nonumber\\
&\cdot
\prod_{\substack{1 \le i \le m \\ 1 \le k \le n}}
\frac{[-x_i-y_k-(\alpha_i-1)\delta]_{\alpha_i-\mu_i}}
{[-x_i-y_k-(\alpha_i+\beta_k-1)\delta]_{\alpha_i-\mu_i}}
\frac{[-x_i+y_k-(\alpha_i-\beta_k)\delta]_{\alpha_i-\mu_i}}
{[-x_i+y_k-\alpha_i\delta]_{\alpha_i-\mu_i}}\nonumber \\
&=
\sum_{0 \le \kappa \le \lambda \le  \beta}
(-1)^{|\kappa|+|\beta-\lambda|}\prod_{k=1}^n
A^-(y_k;\bs{b})_{\kappa_k} 
A^+(y_k+\lambda_k\delta;\bs{b})_{\beta_k -\lambda_k} C_{{\lambda_k}-{\kappa_k}}(y_k+\kappa_k\delta;\bs{b} | \bs{c}) \nonumber\\
&\cdot
\prod_{k=1}^n
\frac{[2y_k+2(\kappa_k-1)\delta]}{[2y_k-2\delta]}
\frac{[2y_k+2\lambda_k\delta]}{[2y_k+2\beta_k\delta]}\nonumber\\
&\cdot \prod_{1 \le k<l \le n}
\frac{[(y_k+\kappa_k\delta-\delta) \pm (y_l+\kappa_l\delta -\delta)]}
{[(y_k-\delta) \pm (y_l-\delta)]}
\frac{[(y_k+\lambda_k\delta) \pm (y_l+\lambda_l\delta)]}
{[(y_k+\beta_k \delta) \pm (y_l+\beta_l\delta)]}\nonumber\\
&\cdot \prod_{1 \le k, l \le n}
\frac{[(y_k-\delta) \pm (y_l+\beta_l\delta)]}
{[(y_k-\delta) \pm (y_l+\lambda_l \delta)]}
\frac{[(y_k+\kappa_k\delta-\delta) \pm (y_l+\lambda_l\delta)]}
{[(y_k+\kappa_k\delta-\delta) \pm (y_l+\beta_l\delta)]}
\frac{[y_k+ y_l-2\delta]_{\kappa_k}}
{[y_k+y_l+(\beta_l -1)\delta]_{\kappa_k}}
\frac{[y_k- y_l-\beta_l \delta]_{\kappa_k}}
{[y_k-y_l+ \delta]_{\kappa_k}}\nonumber\\
&\cdot \prod_{1 \le k, l \le n}
\frac{[-y_k- y_l-(\beta_k+\beta_l)\delta]_{\beta_k-\lambda_k}}
{[-y_k- y_l-(\beta_k-1)\delta]_{\beta_k-\lambda_k}}
\frac{[-y_k+ y_l-\beta_k\delta]_{\beta_k-\lambda_k}}
{[-y_k + y_l + (\beta_l-\beta_k+1)\delta]_{\beta_k-\lambda_k}}\nonumber\\
&\cdot
\prod_{\substack{1 \le k \le n \\ 1 \le i \le m}}
\frac{[y_k+x_i+(\alpha_i-1)\delta]_{\kappa_k}}
{[y_k+x_i-\delta]_{\kappa_k}}
\frac{[y_k-x_i]_{\kappa_k}}
{[y_k-x_i-\alpha_i\delta]_{\kappa_k}} \nonumber\\
&\cdot
\prod_{\substack{1 \le k \le n \\ 1 \le i \le m}}
\frac{[-y_k-x_i-(\beta_k-1)\delta]_{\beta_k-\lambda_k}}
{[-y_k-x_i-(\beta_k+\alpha_i-1)\delta]_{\beta_k-\lambda_k}}
\frac{[-y_k+x_i-(\beta_k-\alpha_i)\delta]_{\beta_k-\lambda_k}}
{[-y_k+x_i-\beta_k\delta]_{\beta_k-\lambda_k}},
\end{align*}
where 
\begin{equation}\label{eq:def C}
C_{\sigma}(z;\bs{a} | \bs{c})=
\sum_{\substack{0 \le r \le \sigma/2}}
(-1)^{r}  
\sum_{(\xi_{p})_{p}}
\prod_{p=1}^{r}
A^+(z+(\xi_{p}-1)\delta;\bs{a})A^-(z+\xi_{p}\delta;\bs{a})
\prod_{\substack{k=0 \\ k \neq \xi_{p}-1, \xi_{p}}}^{\sigma-1}
A^0(z+k \delta;\bs{a} | \bs{c}).
\end{equation}
Here, 
the second summation is taken over all sequences 
$(\xi_{p})_{1 \le p \le r}$ of positive integers satisfying the following conditions\,$:$
\begin{equation*}
\begin{split}
0 < \xi_{1} <  \cdots < \xi_{r} < \sigma, \quad
\xi_{p+1}-\xi_{p}\ge 2 \ (p=1, \ldots, r-1).
\end{split}
\end{equation*}
\end{theo}
We remark that $C_{\sigma}(z;\bs{a} | \bs{c})$ has the following determinant formula:
\begin{equation}\label{eq:det of C}
C_{\sigma}(z;\bs{a} | \bs{c})
 =\det
  \begin{pmatrix}
    A^0_0 & A^-_{1} &  \\
    A^+_{0} & A^0_{1} & A^-_{2} \\
          & A^+_{1} & A^0_{2} &\ddots \\
          &       & \ddots & \ddots & A^-_{\sigma-1}\\
          &       &        & A^+_{\sigma-2}  & A^0_{\sigma-1}\\
  \end{pmatrix},
\end{equation}
where $A^{\epsilon}_\xi = A^{\epsilon}(z+\xi\delta;\bs{a})\ (\epsilon \in \{\pm, 0 \})$
and we have omitted the parameters $\bs{c}$ in $A^0(z;\bs{a}|\bs{c})$.
This can be shown from the fact that both the right-hand sides of \eqref{eq:def C}
and \eqref{eq:det of C} satisfy the following three-term recurrence relation:
\begin{align*}
    F_{\sigma+2}&=F_{\sigma+1}A_{\sigma+1}^0-F_{\sigma}A_{\sigma}^+A_{\sigma+1}^- \quad (\sigma \ge 0)
\end{align*}
with the initial conditions 
\begin{equation*}
  F_{0}=1,\qquad
  F_{1}=A^0_{0}.
\end{equation*}
\subsection{Case where $a_\ph=a_\pa+\delta$}
When two of the $a$ parameters differ by $\delta$,
the transformation formula of Theorem \ref{Theo:BC multi1}
can be generalized further to the case $|\alpha| \neq |\beta|$.
We first consider to generalize the case $r=N$ of Theorem \ref{Theo:BC set}.
\begin{theo}\label{Theo:BC set M not N}
Let $M$ and $N$ be two non-negative integers with $M \ge N$.
For a set of complex parameters  $\bs{a} = (a_\pa, a_\pb, \ldots, a_\ph)$,
we assume the balancing condition $\sum_{p=\pa}^\ph a_p = (4-2M+2N)\delta$ 
and $a_\ph= a_\pa + \delta$.
We define the parameters $\bs{b} =(b_\pa, b_\pb, \ldots, b_\ph)$ by $b_\pa= \delta-a_\ph,
b_\ph=\delta-a_\pa$ and $b_p = \delta - a_p\ (p= \pb, \ldots, \pg)$.
For two sets of variables $z=(z_1, \ldots, z_M)$ and $w=(w_1, \ldots, w_N)$,
the following identity holds\,$:$
\begin{align}
&\sum_{I_+ \cup I_- \cup I_0 = \{1, \ldots, M \}}
\prod_{i \in I_+} A^+(z_i; \bs{a}) \prod_{i \in I_0} A^0(z_i;\bs{a} | \bs{a_\pa})
\prod_{i \in I_-} A^-(z_i;\bs{a})\nonumber\\
&\cdot
\prod_{\{i, j \} \subset I_+} \frac{[z_i +z_j +2\delta]}{[z_i +z_j]}
\prod_{\{i, j \} \subset I_-} \frac{[z_i +z_j -2\delta]}{[z_i +z_j]}
\prod_{\substack{i \in I_+ \\ j \in I_-} }\frac{[z_i - z_j +2\delta]}{[z_i - z_j]} \nonumber\\
&\cdot
\prod_{\substack{i \in I_+ \\ j \in I_0} }
\frac{[z_i + \delta \pm z_j]}{[z_i \pm z_j]}
\prod_{\substack{i \in I_- \\ j \in I_0} }
\frac{[z_i - \delta \pm z_j]}{[z_i \pm z_j]}
\prod_{\substack{i \in I_+ \\ 1 \le k \le N}}
\frac{[z_i \pm w_k]}{[z_i + \delta \pm w_k]}
\prod_{\substack{i \in I_- \\ 1 \le k \le N}}
\frac{[z_i \pm w_k]}{[z_i - \delta \pm w_k]} \nonumber\\
&=
\prod_{p=\pb}^{\pg}[a_\pa +a_p]_{M-N}
\sum_{K_+ \cup K_- \cup K_0 = \{1, \ldots, N \}}
\prod_{k \in K_+} A^+(w_k;\bs{b}) \prod_{k \in K_0} A^0(w_k;\bs{b} | \bs{b_\pa})
\prod_{k \in K_-} A^-(w_k; \bs{b})\nonumber\\
&\cdot
\prod_{\{k, l \} \subset K_+} \frac{[w_k +w_l +2\delta]}{[w_k+w_l]} 
\prod_{\{k, l \} \subset K_-} \frac{[w_k+w_l- 2\delta]}{[w_k+w_l]}
\prod_{\substack{k \in K_+ \\ l \in K_-}}\frac{[w_k-w_l +2\delta]}{[w_k-w_l]} \nonumber\\
&\cdot
\prod_{\substack{k \in K_+ \\ l \in K_0}}
\frac{[w_k+\delta \pm w_l ]}{[w_k \pm w_l]}
\prod_{\substack{k \in K_- \\ l \in K_0}}
\frac{[w_k-\delta \pm w_l ]}{[w_k \pm w_l]} 
\prod_{\substack{k \in K_+ \\ 1 \le i \le M}}
\frac{[w_k \pm z_i]}{[w_k +\delta \pm z_i]}
\prod_{\substack{k \in K_- \\ 1 \le i \le M}}
\frac{[w_k \pm z_i]}{[w_k -\delta \pm z_i]},
\label{eq:BC set M not N}
\end{align}
where $\bs{a_\pa}=(a_\pa, a_\pa, a_\pa, a_\pa)$ and $\bs{b_\pa}=(b_\pa, b_\pa, b_\pa, b_\pa)$.
\end{theo}
\begin{proof}
We begin with the duality transformation formula \eqref{eq:BC set N2} of type $BC$ on subsets.
Setting $I=\{1, \ldots, N, N+1, \ldots,  M \}$
and $a_\pb= a_\pa-\delta, c_r =b_\pa\ (r=0, 1, 2, 3)$,
we consider the partial principal specialization $(w_{N+1}, \ldots, w_M) =(b_\pa, b_\pa+\delta, \ldots, b_\pa+(M-N-1)\delta)$
of a single block.
By the same argument as in Subsection \ref{subsection:Duality transformation of type $BC$ on multi-indices},
the term corresponding to $\epsilon=(\epsilon_1, \ldots, \epsilon_M)$ in the right-hand side becomes zero 
if the sign sequence $(\epsilon_{N+1}, \ldots, \epsilon_M)$
includes some of the patterns $+0, 0-$ and $+*-$.
Further, we find that the term with $\epsilon_{N+1} \in \{0, - \}$
or $\epsilon_{N+2}=-$ vanishes by the definition of $A^0(x;\bs{b}|\bs{b_\pa})$ and $A^-(x;\bs{b})$.
Hence the corresponding term vanishes unless $(\epsilon_{N+1}, \epsilon_{N+2})=(+, +)$.
Since the patterns $+0$ and $+\! *\! -$ are not allowed,
only the terms with $(\epsilon_{N+1}, \epsilon_{N+2}, \ldots, \epsilon_{M})=(+, +, \ldots, +)$ remain.
Multiplying the both sides by $\prod_{i=1}^M\frac{[z_i \pm (b_\pa+(M-N)\delta)]}{[z_i \pm b_\pa]}$,
we compute the non-zero terms.
We first compute by using $b_\pa =\delta-a_\pa=-a_\pb$
\begin{align*}
&A_r^0(z_i;\bs{a} | \bs{b_\pa})
\frac{[z_i \pm (b_\pa+(M-N)\delta)]}{[z_i \pm b_\pa]}\nonumber\\
&=\epsilon_r e((\delta - \tfrac{\omega_r}{2} - \tfrac{1}{2} {\textstyle \sum_{p=\pa}^\ph} a_p)\eta_r)
\frac{[z_i \pm (b_\pa+(M-N)\delta)]\prod_{p=\pa}^\ph[\frac{1}{2}(\omega_r-\delta)+a_p]}
{2[\frac{1}{2}(\omega_r-\delta) \pm z_i][\frac{1}{2}(\omega_r-\delta) \pm b_\pa]}\nonumber\\
&=\epsilon_r e((\delta - \tfrac{\omega_r}{2} - \tfrac{1}{2} {\textstyle \sum_{p=\pa}^\ph} a_p)\eta_r)
\frac{[z_i \pm (b_\pa+(M-N)\delta)]\prod_{p=\pa}^\ph[\frac{1}{2}(\omega_r-\delta)+a_p]}
{2[\frac{1}{2}(\omega_r-\delta) \pm z_i][\frac{1}{2}(\omega_r+\delta) -a_\pa]
[\frac{1}{2}(\omega_r-\delta) +a_\pb]}\nonumber\\
&=\epsilon_r e((\delta - \tfrac{\omega_r}{2} - \tfrac{1}{2} {\textstyle \sum_{p=\pa}^\ph} a_p)\eta_r)
\frac{[z_i \pm (b_\pa+(M-N)\delta)]
\prod_{p=\pc}^\ph[\frac{1}{2}(\omega_r-\delta)+a_p]}
{-2\epsilon_r e((\frac{\delta}{2}-a_\pa)\eta_r)[\frac{1}{2}(\omega_r-\delta) \pm z_i]}. 
\end{align*}
Note here that 
\begin{align*}
&\frac{[\tfrac{1}{2}(\omega_r-\delta)+a_\pa-(M-N)\delta]
[\tfrac{1}{2}(\omega_r-\delta)+a_\pb-(M-N)\delta]}
{[\tfrac{1}{2}(\omega_r-\delta)+ b_\pa+(M-N)\delta]
[\tfrac{1}{2}(\omega_r-\delta) - b_\pa-(M-N)\delta)]}  \nonumber\\
&=\frac{[\tfrac{1}{2}(\omega_r-\delta)+a_\pa-(M-N)\delta]}
{[\tfrac{1}{2}(\omega_r-\delta)+ \delta-a_\pa+(M-N)\delta]} \nonumber\\
&=
\epsilon_r e((-\tfrac{\delta}{2}+a_\pa -(M-N)\delta) \eta_r)
\frac{[\tfrac{1}{2}(\omega_r-\delta)+a_\pa-(M-N)\delta]}
{[\tfrac{1}{2}(-\omega_r +\delta)-a_\pa +(M-N)\delta]} \nonumber\\
&=-\epsilon_r e((-\tfrac{\delta}{2}+a_\pa -(M-N)\delta) \eta_r).
\end{align*}
Hence we have
\begin{align*}
&A_r^0(z_i;\bs{a} | \bs{b_\pa})
\frac{[z_i \pm (b_\pa+(M-N)\delta)]}{[z_i \pm b_\pa]} \nonumber\\
&=\epsilon_r e((\delta - \tfrac{\omega_r}{2} - \tfrac{1}{2} {\textstyle \sum_{p=\pa}^\ph} a_p+(M-N)\delta)\eta_r)
\frac{[z_i \pm (b_\pa+(M-N)\delta)]
\prod_{p=\pc}^\ph[\frac{1}{2}(\omega_r-\delta)+a_p]
}
{2[\frac{1}{2}(\omega_r-\delta) \pm z_i]}\nonumber\\
&\cdot
\frac{[\tfrac{1}{2}(\omega_r-\delta)+a_\pa-(M-N)\delta]
[\tfrac{1}{2}(\omega_r-\delta)+a_\pb-(M-N)\delta]}
{[\tfrac{1}{2}(\omega_r-\delta)+ b_\pa+(M-N)\delta]
[\tfrac{1}{2}(\omega_r-\delta) - b_\pa-(M-N)\delta)]}\nonumber \\
&=A^0_r(z_i; a_\pa-(M-N)\delta, a_\pb-(M-N)\delta, a_\pc, \ldots, a_\ph|
b_\pa+(M-N)\delta).
\end{align*}
Therefore we obtain
\begin{align*}
&(LHS) =
\sum_{I_+ \cup I_- \cup I_0 = \{1, \ldots, M \}}
\prod_{i \in I_+} A^+(z_i; \bs{a}) 
\frac{[z_i- b_\pa-(M-N)\delta][z_i- b_\pa-(M-N-1)\delta]}
{[z_i-b_\pa][z_i-b_\pa+\delta]} \nonumber\\
&\cdot
\prod_{i \in I_0} A^0(z_i;\bs{a} | \bs{b_\pa})
\frac{[z_i \pm (b_\pa+(M-N)\delta)]}{[z_i \pm b_\pa]}\nonumber\\
&\cdot
\prod_{i \in I_-} A^-(z_i;\bs{a})
\frac{[z_i+ b_\pa+(M-N)\delta][z_i+ b_\pa+(M-N-1)\delta]}
{[z_i+b_\pa][z_i+b_\pa-\delta]}\nonumber\\
&\cdot
\prod_{\{i, j \} \subset I_+} \frac{[z_i +z_j +2\delta]}{[z_i +z_j]}
\prod_{\{i, j \} \subset I_-} \frac{[z_i +z_j -2\delta]}{[z_i +z_j]}
\prod_{\substack{i \in I_+ \\ j \in I_-} }\frac{[z_i - z_j +2\delta]}{[z_i - z_j]}\nonumber \\
&\cdot
\prod_{\substack{i \in I_+ \\ j \in I_0} }
\frac{[z_i + \delta \pm z_j]}{[z_i \pm z_j]}
\prod_{\substack{i \in I_- \\ j \in I_0} }
\frac{[z_i - \delta \pm z_j]}{[z_i \pm z_j]}
\prod_{\substack{i \in I_+ \\ 1 \le k \le N}}
\frac{[z_i \pm w_k]}{[z_i + \delta \pm w_k]}
\prod_{\substack{i \in I_- \\ 1 \le k \le N}}
\frac{[z_i \pm w_k]}{[z_i - \delta \pm w_k]}\nonumber\\
&=
\sum_{I_+ \cup I_- \cup I_0 = \{1, \ldots, M \}}
\prod_{i \in I_+} A^+(z_i; a_\pa-(M-N)\delta, a_\pb-(M-N)\delta, a_\pc, \ldots, a_\ph) \nonumber\\
&\cdot
\prod_{i \in I_0} A^0(z_i; a_\pa-(M-N)\delta, a_\pb-(M-N)\delta, a_\pc, \ldots, a_\ph|
\bs{\widetilde{b}_0})\nonumber\\
&\cdot
\prod_{i \in I_-} A^-(z_i;a_\pa-(M-N)\delta, a_\pb-(M-N)\delta, a_\pc, \ldots, a_\ph)
\nonumber\\
&\cdot
\prod_{\{i, j \} \subset I_+} \frac{[z_i +z_j +2\delta]}{[z_i +z_j]}
\prod_{\{i, j \} \subset I_-} \frac{[z_i +z_j -2\delta]}{[z_i +z_j]}
\prod_{\substack{i \in I_+ \\ j \in I_-} }\frac{[z_i - z_j +2\delta]}{[z_i - z_j]} \nonumber\\
&\cdot
\prod_{\substack{i \in I_+ \\ j \in I_0} }
\frac{[z_i + \delta \pm z_j]}{[z_i \pm z_j]}
\prod_{\substack{i \in I_- \\ j \in I_0} }
\frac{[z_i - \delta \pm z_j]}{[z_i \pm z_j]}
\prod_{\substack{i \in I_+ \\ 1 \le k \le N}}
\frac{[z_i \pm w_k]}{[z_i + \delta \pm w_k]}
\prod_{\substack{i \in I_- \\ 1 \le k \le N}}
\frac{[z_i \pm w_k]}{[z_i - \delta \pm w_k]}.
\end{align*}
Here $\bs{\widetilde{b}_\pa} = (\widetilde{b}_\pa, \widetilde{b}_\pa, \widetilde{b}_\pa, \widetilde{b}_\pa)$ and 
$\widetilde{b}_\pa=b_\pa+(M-N)\delta$.
By the same computation, we have
\begin{align*}
&(RHS)=
\prod_{p=\pc}^\ph [b_\pa+b_p]_{M-N}
\sum_{K_+ \cup K_- \cup K_0 = \{1, \ldots, N \}}
\prod_{k \in K_+} A^+(w_k;b_\pa+(M-N)\delta, b_\pb+(M-N)\delta, b_\pc, \ldots, b_\ph) \nonumber\\
&\cdot
\prod_{k \in K_0} A^0(w_k;b_\pa+(M-N)\delta, b_\pb+(M-N)\delta, b_\pc, \ldots, b_\ph |
\bs{\widetilde{b}_0})\nonumber\\
&\cdot
\prod_{k \in K_-} A^-(w_k; b_\pa+(M-N)\delta, b_\pb+(M-N)\delta, b_\pc, \ldots, b_\ph)\nonumber\\
&\cdot
\prod_{\{k, l \} \subset K_+} \frac{[w_k +w_l +2\delta]}{[w_k+w_l]} 
\prod_{\{k, l \} \subset K_-} \frac{[w_k+w_l- 2\delta]}{[w_k+w_l]}
\prod_{\substack{k \in K_+ \\ l \in K_-}}\frac{[w_k-w_l +2\delta]}{[w_k-w_l]}\nonumber\\
&\cdot
\prod_{\substack{k \in K_+ \\ l \in K_0}}
\frac{[w_k+\delta \pm w_l ]}{[w_k \pm w_l]}
\prod_{\substack{k \in K_- \\ l \in K_0}}
\frac{[w_k-\delta \pm w_l ]}{[w_k \pm w_l]} 
\prod_{\substack{k \in K_+ \\ 1 \le i \le M}}
\frac{[w_k \pm z_i]}{[w_k +\delta \pm z_i]}
\prod_{\substack{k \in K_- \\ 1 \le i \le M}}
\frac{[w_k \pm z_i]}{[w_k -\delta \pm z_i]}.
\end{align*}
From this,
we obtain  formula \eqref{eq:BC set M not N} by
replacing $(a_\pa-(M-N)\delta, a_\pb-(M-N)\delta, a_\pc, \ldots, a_\ph),
(b_\pa+(M-N)\delta, b_\pb+(M-N)\delta, b_\pc, \ldots, b_\ph)$
with $(a_\ph, a_\pa, a_\pb, \ldots, a_\pg),
(b_\pa, b_\ph, b_\pb, \ldots, b_\pg)$ respectively.
\end{proof}
As a special case $N=0$ of Theorem \ref{Theo:BC set M not N},
we obtain a summation formula on subsets. 
\begin{coro}\label{Coro:BC summation}
Suppose that the balancing condition $\sum_{p=\pa}^\ph a_p = (4-2M)\delta$ 
is satisfied.
If $a_\ph= a_\pa + \delta$,
the following identity holds\,$:$
\begin{align*}
&\sum_{I_+ \cup I_- \cup I_0 = \{1, \ldots, M \}}
\prod_{i \in I_+} A^+(z_i; \bs{a}) \prod_{i \in I_0} A^0(z_i;\bs{a} | \bs{a_\pa})
\prod_{i \in I_-} A^-(z_i;\bs{a}) \nonumber\\
&\cdot
\prod_{\{i, j \} \subset I_+} \frac{[z_i +z_j +2\delta]}{[z_i +z_j]}
\prod_{\{i, j \} \subset I_-} \frac{[z_i +z_j -2\delta]}{[z_i +z_j]}
\prod_{\substack{i \in I_+ \\ j \in I_-} }\frac{[z_i - z_j +2\delta]}{[z_i - z_j]} \nonumber\\
&\cdot
\prod_{\substack{i \in I_+ \\ j \in I_0} }
\frac{[z_i + \delta \pm z_j]}{[z_i \pm z_j]}
\prod_{\substack{i \in I_- \\ j \in I_0} }
\frac{[z_i - \delta \pm z_j]}{[z_i \pm z_j]}\nonumber\\
&=
\prod_{p=\pb}^{\pg}[a_0 +a_p]_{M}.
\end{align*}
\end{coro}
Considering the multiple principal specialization of \eqref{eq:BC set M not N}
as in Subsection \ref{subsection:Duality transformation of type $BC$ on multi-indices},
we obtain the following transformation formula. 
\begin{theo}\label{Theo:BC duality transformation M not N}
Let $M$ and $N$ be two non-negative integers with $M \ge N$.
Take two multi-indices $\alpha \in \mathbb{N}^m$ and 
$\beta \in \mathbb{N}^n$ with $|\alpha|=M$ and $|\beta|=N$.
For a set of complex parameters $\bs{a} = (a_\pa, a_\pb, \ldots, a_\ph)$,
we assume the balancing condition $\sum_{p=\pa}^\ph a_p =(4-2M+2N)\delta$
and $a_\ph=a_\pa+\delta$.
We define the parameters $\bs{b} =(b_\pa, b_\pb, \ldots, b_\ph)$ by $b_\pa= \delta-a_\ph,
b_\ph=\delta-a_\pa$ and $b_p = \delta - a_p\ (p= \pb, \ldots, \pg)$.
For two sets of variables $x=(x_1, \ldots, x_m)$ and $y=(y_1, \ldots, y_n)$,
the following identity holds\,$:$
\begin{align*}
&\sum_{0 \le \nu \le \mu \le  \alpha}
(-1)^{|\nu|+|\alpha-\mu|}
\prod_{i=1}^m 
A^-(x_i;\bs{a})_{\nu_i} 
A^+(x_i+\mu_i\delta;\bs{a})_{\alpha_i -\mu_i} C_{{\mu_i}-{\nu_i}}(x_i+\nu_i\delta;\bs{a} | \bs{a_0})  \nonumber\\
&\cdot 
\prod_{i=1}^m
\frac{[2x_i+2(\nu_i-1)\delta]}{[2x_i-2\delta]}
\frac{[2x_i+2\mu_i\delta]}{[2x_i+2\alpha_i\delta]}\nonumber\\
&\cdot \prod_{1 \le i< j \le m}
\frac{[(x_i+\nu_i\delta-\delta) \pm (x_j+\nu_j\delta -\delta)]}
{[(x_i-\delta) \pm (x_j-\delta)]}
\frac{[(x_i+\mu_i\delta) \pm (x_j+\mu_j\delta)]}
{[(x_i+\alpha_i \delta) \pm (x_j+\alpha_j\delta)]}\nonumber\\
&\cdot \prod_{1 \le i, j \le m}
\frac{[(x_i-\delta) \pm (x_j+\alpha_j\delta)]}
{[(x_i-\delta) \pm (x_j+\mu_j \delta)]}
\frac{[(x_i+\nu_i\delta-\delta) \pm (x_j+\mu_j\delta)]}
{[(x_i+\nu_i\delta-\delta) \pm (x_j+\alpha_j\delta)]}
\frac{[x_i+ x_j-2\delta]_{\nu_i}}
{[x_i+x_j+(\alpha_j -1)\delta]_{\nu_i}}
\frac{[x_i- x_j-\alpha_j \delta]_{\nu_i}}
{[x_i-x_j+ \delta]_{\nu_i}}\nonumber\\
&\cdot \prod_{1 \le i, j \le m}
\frac{[-x_i- x_j-(\alpha_i+\alpha_j)\delta]_{\alpha_i-\mu_i}}
{[-x_i- x_j-(\alpha_i-1)\delta]_{\alpha_i-\mu_i}}
\frac{[-x_i+ x_j-\alpha_i\delta]_{\alpha_i-\mu_i}}
{[-x_i + x_j + (\alpha_j-\alpha_i+1)\delta]_{\alpha_i-\mu_i}}\nonumber\\
&\cdot
\prod_{\substack{1 \le i \le m \\ 1 \le k \le n}}
\frac{[x_i+y_k+(\beta_k-1)\delta]_{\nu_i}}
{[x_i+y_k-\delta]_{\nu_i}}
\frac{[x_i-y_k]_{\nu_i}}
{[x_i-y_k-\beta_k\delta]_{\nu_i}} \nonumber\\
&\cdot
\prod_{\substack{1 \le i \le m \\ 1 \le k \le n}}
\frac{[-x_i-y_k-(\alpha_i-1)\delta]_{\alpha_i-\mu_i}}
{[-x_i-y_k-(\alpha_i+\beta_k-1)\delta]_{\alpha_i-\mu_i}}
\frac{[-x_i+y_k-(\alpha_i-\beta_k)\delta]_{\alpha_i-\mu_i}}
{[-x_i+y_k-\alpha_i\delta]_{\alpha_i-\mu_i}}\nonumber \\
&=
\prod_{p=\pb}^{\pg}[a_\pa+a_p]_{M-N}
\sum_{0 \le \kappa \le \lambda \le  \beta}
(-1)^{|\kappa|+|\beta-\lambda|}\prod_{k=1}^n
A^-(y_k;\bs{b})_{\kappa_k} 
A^+(y_k+\lambda_k\delta;\bs{b})_{\beta_k -\lambda_k} C_{{\lambda_k}-{\kappa_k}}(y_k+\kappa_k\delta;\bs{b} | \bs{b_0}) \nonumber\\
&\cdot
\prod_{k=1}^n
\frac{[2y_k+2(\kappa_k-1)\delta]}{[2y_k-2\delta]}
\frac{[2y_k+2\lambda_k\delta]}{[2y_k+2\beta_k\delta]}\nonumber\\
&\cdot \prod_{1 \le k<l \le n}
\frac{[(y_k+\kappa_k\delta-\delta) \pm (y_l+\kappa_l\delta -\delta)]}
{[(y_k-\delta) \pm (y_l-\delta)]}
\frac{[(y_k+\lambda_k\delta) \pm (y_l+\lambda_l\delta)]}
{[(y_k+\beta_k \delta) \pm (y_l+\beta_l\delta)]}\nonumber\\
&\cdot \prod_{1 \le k, l \le n}
\frac{[(y_k-\delta) \pm (y_l+\beta_l\delta)]}
{[(y_k-\delta) \pm (y_l+\lambda_l \delta)]}
\frac{[(y_k+\kappa_k\delta-\delta) \pm (y_l+\lambda_l\delta)]}
{[(y_k+\kappa_k\delta-\delta) \pm (y_l+\beta_l\delta)]}
\frac{[y_k+ y_l-2\delta]_{\kappa_k}}
{[y_k+y_l+(\beta_l -1)\delta]_{\kappa_k}}
\frac{[y_k- y_l-\beta_l \delta]_{\kappa_k}}
{[y_k-y_l+ \delta]_{\kappa_k}}\nonumber\\
&\cdot \prod_{1 \le k, l \le n}
\frac{[-y_k- y_l-(\beta_k+\beta_l)\delta]_{\beta_k-\lambda_k}}
{[-y_k- y_l-(\beta_k-1)\delta]_{\beta_k-\lambda_k}}
\frac{[-y_k+ y_l-\beta_k\delta]_{\beta_k-\lambda_k}}
{[-y_k + y_l + (\beta_l-\beta_k+1)\delta]_{\beta_k-\lambda_k}}\nonumber\\
&\cdot
\prod_{\substack{1 \le k \le n \\ 1 \le i \le m}}
\frac{[y_k+x_i+(\alpha_i-1)\delta]_{\kappa_k}}
{[y_k+x_i-\delta]_{\kappa_k}}
\frac{[y_k-x_i]_{\kappa_k}}
{[y_k-x_i-\alpha_i\delta]_{\kappa_k}} \nonumber\\
&\cdot
\prod_{\substack{1 \le k \le n \\ 1 \le i \le m}}
\frac{[-y_k-x_i-(\beta_k-1)\delta]_{\beta_k-\lambda_k}}
{[-y_k-x_i-(\beta_k+\alpha_i-1)\delta]_{\beta_k-\lambda_k}}
\frac{[-y_k+x_i-(\beta_k-\alpha_i)\delta]_{\beta_k-\lambda_k}}
{[-y_k+x_i-\beta_k\delta]_{\beta_k-\lambda_k}},
\end{align*}
where $\bs{a_\pa}=(a_\pa, a_\pa, a_\pa, a_\pa), \bs{b_\pa}=(b_\pa, b_\pa, b_\pa, b_\pa)$
and $C_{\sigma}(z; \bs{a}| \bs{c})$ is defined in \eqref{eq:def C}.
\end{theo}
If we set $\beta=0$ in Theorem \ref{Theo:BC duality transformation M not N},
we obtain the following summation formula.
\begin{coro}\label{Coro:BC summation for multiple}
Let $M$ be a non-negative integer.
Take a multi-index $\alpha \in \mathbb{N}^m$ with $|\alpha|=M$.
For a set of complex parameters $\bs{a} = (a_\pa, a_\pb, \ldots, a_\ph)$,
we assume the balancing condition $\sum_{p=\pa}^\ph a_p =(4-2M)\delta$
and $a_\ph=a_\pa+\delta$.
For a set of variables $x=(x_1, \ldots, x_m)$,
the following identity holds\,$:$
\begin{align*}
&\sum_{0 \le \nu \le \mu \le  \alpha}
(-1)^{|\nu|+|\alpha-\mu|}
\prod_{i=1}^m 
A^-(x_i;\bs{a})_{\nu_i} 
A^+(x_i+\mu_i\delta;\bs{a})_{\alpha_i -\mu_i} C_{{\mu_i}-{\nu_i}}(x_i+\nu_i\delta;\bs{a} | \bs{a_0})  \nonumber\\
&\cdot 
\prod_{i=1}^m
\frac{[2x_i+2(\nu_i-1)\delta]}{[2x_i-2\delta]}
\frac{[2x_i+2\mu_i\delta]}{[2x_i+2\alpha_i\delta]}\nonumber\\
&\cdot \prod_{1 \le i< j \le m}
\frac{[(x_i+\nu_i\delta-\delta) \pm (x_j+\nu_j\delta -\delta)]}
{[(x_i-\delta) \pm (x_j-\delta)]}
\frac{[(x_i+\mu_i\delta) \pm (x_j+\mu_j\delta)]}
{[(x_i+\alpha_i \delta) \pm (x_j+\alpha_j\delta)]}\nonumber\\
&\cdot \prod_{1 \le i, j \le m}
\frac{[(x_i-\delta) \pm (x_j+\alpha_j\delta)]}
{[(x_i-\delta) \pm (x_j+\mu_j \delta)]}
\frac{[(x_i+\nu_i\delta-\delta) \pm (x_j+\mu_j\delta)]}
{[(x_i+\nu_i\delta-\delta) \pm (x_j+\alpha_j\delta)]}
\frac{[x_i+ x_j-2\delta]_{\nu_i}}
{[x_i+x_j+(\alpha_j -1)\delta]_{\nu_i}}
\frac{[x_i- x_j-\alpha_j \delta]_{\nu_i}}
{[x_i-x_j+ \delta]_{\nu_i}}\nonumber\\
&\cdot \prod_{1 \le i, j \le m}
\frac{[-x_i- x_j-(\alpha_i+\alpha_j)\delta]_{\alpha_i-\mu_i}}
{[-x_i- x_j-(\alpha_i-1)\delta]_{\alpha_i-\mu_i}}
\frac{[-x_i+ x_j-\alpha_i\delta]_{\alpha_i-\mu_i}}
{[-x_i + x_j + (\alpha_j-\alpha_i+1)\delta]_{\alpha_i-\mu_i}}\nonumber\\
&=
\prod_{p=\pb}^{\pg}[a_\pa+a_p]_{M},
\end{align*}
where $\bs{a_\pa}=(a_\pa, a_\pa, a_\pa, a_\pa)$.
\end{coro}
It would be an interesting problem to clarify how Corollaries \ref{Coro:BC summation}
and \ref{Coro:BC summation for multiple}  are related to the Ruijsenaars\,--\,van Diejen difference operators in higher dimensions.

\section*{Acknowledgment}
This research is partially supported by Grant-in-Aid for Scientific Research (C) 25400026 
and (B) 15H03626.

\end{document}